\RequirePackage{fix-cm}
\documentclass[envcountsame,envcountsect]{svjour3}
\journalname{Mathematical Programming}

\smartqed

\usepackage{graphicx}
\usepackage{mathtools,amssymb,amsfonts}
\usepackage{algorithm}
\usepackage{algorithmicx}
\usepackage{algpseudocode}
\usepackage{booktabs}
\usepackage{mathrsfs}
\usepackage{tikz}
\usetikzlibrary{arrows.meta}
\usepackage{microtype}
\usepackage{cite}
\usepackage{hyperref}
\usepackage{orcidlink}
\usepackage{marvosym}
\providecommand{\doi}[1]{\url{https://doi.org/#1}}
\usepackage[capitalise,nameinlink,noabbrev]{cleveref}

\crefname{theorem}{Theorem}{Theorems}
\Crefname{theorem}{Theorem}{Theorems}
\crefname{proposition}{Proposition}{Propositions}
\Crefname{proposition}{Proposition}{Propositions}
\crefname{lemma}{Lemma}{Lemmas}
\Crefname{lemma}{Lemma}{Lemmas}
\crefname{corollary}{Corollary}{Corollaries}
\Crefname{corollary}{Corollary}{Corollaries}
\crefname{example}{Example}{Examples}
\Crefname{example}{Example}{Examples}
\crefname{remark}{Remark}{Remarks}
\Crefname{remark}{Remark}{Remarks}
\crefname{definition}{Definition}{Definitions}
\Crefname{definition}{Definition}{Definitions}

\hypersetup{
  hidelinks,
  hypertexnames=false,
  bookmarksdepth=2
}

\counterwithin{algorithm}{section}

\newcommand{\R}{\mathbb{R}}
\newcommand{\norm}[1]{\left\lVert #1\right\rVert}
\newcommand{\dd}{\mathop{}\!\mathrm{d}}
\newcommand{\spec}{\operatorname{spec}}
\newcommand{\diag}{\operatorname{diag}}
\newcommand{\Lip}{\operatorname{Lip}}
\newcommand{\Id}{\mathrm{Id}}
\newcommand{\prox}{\operatorname{prox}}
\newcommand{\Rot}{\operatorname{Rot}}
\newcommand{\clr}{\operatorname{clr}}
\newcommand{\dist}{\operatorname{dist}}
\newcommand{\classF}{\mathscr{F}}
\newcommand{\X}{\mathcal{X}}
\newcommand{\Bcert}{\mathcal{B}_{\mathrm{cert}}}
\newcommand{\PPD}{\ensuremath{\mathsf{PPD}}}
\newcommand{\Cacc}{C_{\scriptscriptstyle\mathrm{acc}}}
\newcommand{\Cdim}{C_{\scriptscriptstyle\mathrm{dim}}}

\begin{document}

\title{First-Order Optimization under Uniform Nondegeneracy: Geometry, Computation, and Information}
\titlerunning{First-Order Optimization under Uniform Nondegeneracy}

\author{Hua Su \and Lei Zhang\,\orcidlink{0000-0001-9972-2051} \and Jin Zhao}

\authorrunning{Hua Su et al.}

\institute{
Hua Su \at
Beijing International Center for Mathematical Research,
Peking University, Beijing, China \\
\email{suhua@pku.edu.cn}
\and
Lei Zhang \at
School of Mathematical Sciences,
Beijing International Center for Mathematical Research,
Center for Quantitative Biology,
Center for Machine Learning Research,
Peking University, Beijing, China \\
\Letter\ \email{\href{mailto:zhangl@math.pku.edu.cn}{zhangl@math.pku.edu.cn}}
\and
Jin Zhao \at
Academy for Multidisciplinary Studies,
Capital Normal University, Beijing, China \\
\email{zjin@cnu.edu.cn}
}

\date{Received: date / Accepted: date}

\maketitle

\begin{abstract}
Strongly convex minimization and its natural indefinite extension to strongly convex--strongly concave minimax problems combine quantitative control of curvature with a prescribed curvature orientation. We disentangle these two roles by retaining uniform nondegeneracy alone: curvature remains uniformly separated from zero but may have either sign, with no prescribed positive--negative splitting. Surprisingly, a large part of the familiar theory nevertheless re-emerges. We first derive an intrinsic formulation through first-order secant inequalities, making the gradient on $\R^d$ a global bi-Lipschitz homeomorphism and yielding a unique stationary point. We then pair signed Moreau envelopes to construct a smooth scalar merit that recovers the missing descent geometry at both zeroth and first order, and the resulting paired proximal descent method achieves global linear convergence with dimension-free first-order oracle complexity. Meanwhile, we show that this tractability can break down on restricted domains: merely assuming the existence of a stationary point in the domain may lead to the curse of dimensionality, even with access to an infinite-order oracle. To overcome this information barrier, we introduce certified feasibility, an observable localization condition that enables feasible continuation. Together, these results establish a first-order optimization theory under uniform nondegeneracy that spans geometry, computation, and information.

\end{abstract}

\section{Introduction}

Smooth strongly convex minimization provides a basic setting in which curvature control leads to global first-order tractability \cite{Beck2017,Nesterov2018}. For a twice differentiable objective $E\in C^2(\R^d,\R)$, write $g:=\nabla E$ and $H:=\nabla^2E$, and let $\spec$ denote the spectrum of a symmetric matrix. The spectral condition
\begin{equation*}
\spec(H(x))\subseteq[\mu,L],
\qquad 0<\mu\le L,
\end{equation*}
makes the gradient $g$ both $L$-Lipschitz and $\mu$-strongly monotone. The minimizer $x^\star$ is therefore unique, the objective gap $E(x)-E(x^\star)$ provides a natural measure of progress toward it, and standard first-order methods converge globally at dimension-free rates governed by $L/\mu$.

As a natural indefinite extension, strongly convex--strongly concave (SCSC) minimax optimization allows curvature of both signs while prescribing their orientation through the variable splitting. For an SCSC problem $\min_u\max_v E(u,v)$, the saddle operator $\bigl(\nabla_uE(u,v),-\nabla_vE(u,v)\bigr)$ is strongly monotone. Together with Lipschitz continuity, this ensures a unique saddle point and supports a dimension-free first-order theory analogous to smooth strongly convex minimization \cite{Nemirovski2004,LinJinJordan2020}.

Thus the curvature conditions underlying these classical models serve two distinct roles: they keep curvature uniformly separated from zero and prescribe an orientation under which the relevant first-order map is strongly monotone. In this paper, we disentangle these roles and ask how much of the preceding global first-order theory follows from uniform nondegeneracy itself, without any prescribed curvature orientation.

Specifically, given spectral endpoints $L_-\ge\mu_->0$ and $L_+\ge\mu_+>0$, we consider the uniformly nondegenerate function class
\begin{equation}
\classF_{\mu_\pm,L_\pm}(\R^d)
:=\left\{E\in C^2(\R^d):
\spec(H(x))\subseteq[-L_-,-\mu_-]\cup[\mu_+,L_+],
\;\forall x\in\R^d\right\}.
\label{eq:smooth-class-definition}
\end{equation}
When the two spectral branches have the same endpoints, we write $\classF_{\mu,L}(\R^d)$.

The class contains strongly convex, strongly concave, and SCSC objectives, but is strictly larger: its gradient need not be monotone, and its indefinite members need not admit any fixed convex--concave splitting. It is therefore no longer evident that dimension-free tractability should persist. Nevertheless, we show that a surprisingly large part of the global first-order theory familiar from these classical models re-emerges under uniform nondegeneracy alone.

\subsection{Our contributions}

We make three principal contributions.
\begin{enumerate}
\item We formulate uniform nondegeneracy intrinsically through a pair of first-order secant inequalities. This geometry makes the gradient a global bi-Lipschitz homeomorphism on $\R^d$ and yields reciprocal Legendre duality together with a signed Polyak--\L{}ojasiewicz (PL) inequality. We further decompose the class by stationary-point index; its intermediate components include objectives with no fixed convex--concave splitting and therefore go strictly beyond familiar SCSC models.

\item We construct a paired Moreau merit $V\in C^1(\R^d)$ satisfying
\begin{equation*}
V\asymp\norm{g}^2,
\qquad
\norm{\nabla V}\asymp\norm{g},
\end{equation*}
These equivalences imply a global PL inequality. The resulting paired proximal descent method uses only first-order queries and, from arbitrary initialization, finds an $\varepsilon$-stationary point within
\begin{equation*}
O\!\left(\frac{L_{\max}^2}{\mu_+\mu_-}\log\frac{\norm{g(x_0)}}\varepsilon\right)
\end{equation*}
queries, independently of the dimension, where $L_{\max}:=\max\{L_+,L_-\}$.

\item We prove that stationary-point computation under promised feasibility alone suffers the curse of dimensionality: even on the unit ball, its deterministic fixed-accuracy complexity is $\exp(\Theta(d))$ under an infinite-order oracle. We then introduce certified feasibility, an observable localization condition under which feasible continuation again has dimension-free complexity.
\end{enumerate}

\subsection{Related work}
Beyond the strongly convex and SCSC benchmarks recalled above, existing methods retain or reconstruct a usable descent--ascent orientation in two different ways. One route retains a prescribed minimax splitting, as in $J$-symmetric quasi-Newton methods \cite{AslLuYang2024}. Under positive interaction dominance, the associated saddle envelope becomes SCSC, yielding linear convergence of a damped exact proximal-point method \cite{GrimmerLuWorahMirrokni2023}; the same condition also supports linear convergence of damped extragradient \cite{HajizadehLuGrimmer2024}. When no curvature splitting is supplied, gentlest-ascent \cite{EZhou2011}, dimer \cite{HenkelmanJonsson1999}, and high-index saddle dynamics \cite{YinZhangZhang2019,SuWangZhangZhaoZheng2025} reconstruct it by estimating the Hessian's unstable eigenspace, thereby recovering the descent--ascent structure supplied a priori by classical minimax models. By contrast, our framework requires no descent--ascent splitting, whether prescribed by the model or reconstructed from Hessian information.

The stationary equation $g(x)=0$ has also been widely studied as a symmetric nonlinear equation. Oren's planar quasi-Newton method terminates within $d$ iterations on nonsingular symmetric indefinite quadratics \cite{Oren1984}. For nonlinear equations, Gauss--Newton--BFGS \cite{LiFukushima1999} and nonmonotone BFGS \cite{Zhou2022} provide global asymptotic convergence under bounded and uniformly nonsingular symmetric Jacobians on a residual-level-set neighborhood. Grapiglia and Chorobura establish finite evaluation bounds for derivative-free solvers under Lipschitz-Jacobian and bounded-level-set assumptions, with logarithmic complexity in the strongly monotone regime \cite{GrapigliaChorobura2021}. Our whole-space theorem instead gives a dimension-free, endpoint-explicit first-order oracle bound from arbitrary initialization under uniform nondegeneracy, without a Jacobian-continuity modulus or strong monotonicity.

A standard way to obtain a scalar descent objective for such an equation is to minimize the squared residual $\norm{g(x)}^2$. Hamiltonian gradient descent applies this construction to a prescribed minimax operator; its update differentiates that operator, and its linear-convergence analysis uses Lipschitz control of the Jacobian \cite{AbernethyLaiWibisono2021}. Proximal regularization provides another source of smooth scalar functions: the classical Moreau envelope has a gradient represented by a proximal displacement \cite[Sec.~12.4]{BauschkeCombettes2017}. Weighted combinations of envelopes and proximal mappings are developed in proximal-average theory \cite{BauschkeGoebelLucetWang2008,ChenWangPlaniden2020}. Differences of regularized gap functions form D-gap merits for variational inequalities \cite{YamashitaTajiFukushima1997,KanzowFukushima1998}, while differences of Moreau envelopes smooth objectives with a supplied DC decomposition \cite{SunSun2023}. In our setting, the two-sided proximal path produces a merit whose value is equivalent to the squared residual and whose gradient norm is equivalent to the residual itself; these estimates yield global PL geometry and make descent on the merit realizable from first-order queries.

For general smooth nonconvex objectives, Carmon et al. establish oracle lower bounds driven by the requested accuracy \cite{CarmonDuchiHinderSidford2020}, while Hollender and Zampetakis study computational and query complexity when dimension is decoupled from accuracy \cite{HollenderZampetakis2025}. Our restricted-domain theorem concerns a different obstruction: exponential dependence on dimension at fixed accuracy. A particularly close information-based precedent is Sikorski's result, which establishes such dependence for general Lipschitz equations on a cube under a promised zero \cite{Sikorski1984}. We show that this dimensional barrier persists even when the equation is the gradient of a $C^\infty$ uniformly nondegenerate objective with fixed spectral endpoints and index, the stationary point is unique, and each query returns the complete jet.

\subsection{Organization}
\Cref{sec:class} develops the first-order geometry of uniform nondegeneracy. \Cref{sec:proximal} converts that geometry into paired proximal descent and its whole-space query bound. \Cref{sec:query-restrictions} keeps the same intrinsic geometry under domain-restricted oracle access, first exposing the information barrier under promised feasibility and then recovering dimension-free computation through certification. \Cref{sec:conclusion} concludes, and technical proofs are collected in the appendices.

\section{First-order geometry of uniform nondegeneracy}
\label{sec:class}

We begin by asking what remains of the classical first-order geometry when positive definiteness is replaced by uniform nondegeneracy. Under strong convexity, Hessian bounds make the gradient both Lipschitz continuous and strongly monotone. Uniform nondegeneracy retains quantitative curvature control but permits both signs, so the gradient need not be monotone. Moreover, the spectral condition in \eqref{eq:smooth-class-definition} is second order, whereas the oracle reveals only $E(x)$ and $g(x)$. An intrinsic first-order formulation is therefore essential.

\subsection{Secant characterization and metric rigidity}

We first record two facts behind the subsequent intrinsic characterization: a bounded spectral radius is equivalent to a Lipschitz gradient, while, more subtly, a uniform spectral gap around zero yields global metric invertibility without monotonicity. The following lemma makes this precise.
\begin{lemma}[Spectral bounds and gradient geometry]
\label[lemma]{lem:global-gradient-geometry}
Let $f\in C^2(\R^d)$.
\begin{enumerate}
\item\label{item:gradient-lipschitz} The condition $\spec(\nabla^2f(x))\subseteq[-\ell,\ell]$ for every $x$ is equivalent to
\begin{equation}
\norm{\nabla f(x_1)-\nabla f(x_2)}\le\ell\norm{x_1-x_2},
\label{eq:lipschitz-gradient}
\end{equation}
holding for every $x_1,x_2\in\R^d$.

\item\label{item:gradient-co-lipschitz} The condition $\spec(\nabla^2f(x))\subseteq(-\infty,-\nu]\cup[\nu,\infty)$ for every $x$ is equivalent to
\begin{equation}
\norm{\nabla f(x_1)-\nabla f(x_2)}\ge\nu\norm{x_1-x_2},
\label{eq:co-lipschitz-gradient}
\end{equation}
holding for every $x_1,x_2\in\R^d$.
Under these equivalent conditions, $\nabla f$ is a global $C^1$ diffeomorphism and its inverse is $\nu^{-1}$-Lipschitz.
\end{enumerate}
\end{lemma}

\begin{proof}
We first derive the spectral conditions from the secant estimates. For fixed $x,h\in\R^d$, Taylor expansion gives $\nabla f(x+\beta h)-\nabla f(x)=\beta\nabla^2f(x)h+o(\beta)$. Substituting this expansion into \eqref{eq:lipschitz-gradient} and \eqref{eq:co-lipschitz-gradient} yields $\norm{\nabla^2f(x)h}\le\ell\norm{h}$ and $\norm{\nabla^2f(x)h}\ge\nu\norm{h}$, respectively. Since the Hessian is symmetric, these are precisely the two asserted spectral conditions.

We now prove the converse implications. Set $\Delta x:=x_1-x_2$. Under the upper spectral bound, $\norm{\nabla^2f}\le\ell$, and direct integration along the segment proves \eqref{eq:lipschitz-gradient} by
\begin{equation*}
\norm{\nabla f(x_1)-\nabla f(x_2)}
\le\int_0^1\norm{\nabla^2f(x_2+t\Delta x)\Delta x}\dd t
\le\ell\norm{\Delta x},
\end{equation*}

The lower spectral bound is the substantive case. If the Hessian were positive definite, the same segment integration would give the strong-monotonicity estimate
\begin{equation*}
\langle\nabla f(x_1)-\nabla f(x_2),\Delta x\rangle
=\int_0^1\langle\nabla^2f(x_2+t\Delta x)\Delta x,\Delta x\rangle\dd t
\ge\nu\norm{\Delta x}^2.
\end{equation*}
However, this argument breaks down under an indefinite spectral gap: the integrand has no fixed sign, and varying curvature subspaces can produce cancellations. The spectral gap gives instead the pointwise inverse bound
\begin{equation}
\norm{\nabla^2f(x)^{-1}}\le\nu^{-1},
\qquad x\in\R^d.
\label{eq:inverse-hessian-bound}
\end{equation}
This inverse bound does not control the direct Hessian integral along the segment from $x_2$ to $x_1$. We instead invoke the Hadamard--L\'evy theorem: a $C^1$ local diffeomorphism $F:\R^d\to\R^d$ with uniformly bounded $DF(x)^{-1}$ is a global diffeomorphism \cite[Theorem~3.2]{Plastock1974}. Here the spectral gap makes $F=\nabla f$ a local diffeomorphism, while \eqref{eq:inverse-hessian-bound} supplies the uniform bound. Hence $\nabla f$ is globally invertible; a specialized proof of this step is given in \Cref{app:global-inverse}.

We can therefore invert first and integrate in gradient space. Along the segment $\zeta(s):=\nabla f(x_2)+s[\nabla f(x_1)-\nabla f(x_2)]$, one has
\begin{equation*}
D[(\nabla f)^{-1}](p)
=\nabla^2f\left((\nabla f)^{-1}(p)\right)^{-1},
\end{equation*}
so \eqref{eq:inverse-hessian-bound} yields
\begin{equation*}
\begin{aligned}
\norm{x_1-x_2}
&=\norm{\int_0^1D[(\nabla f)^{-1}](\zeta(s))[\nabla f(x_1)-\nabla f(x_2)]\dd s}\\
&\le\nu^{-1}\norm{\nabla f(x_1)-\nabla f(x_2)},
\end{aligned}
\end{equation*}
which is precisely \eqref{eq:co-lipschitz-gradient}. The same estimate shows that $(\nabla f)^{-1}$ is $\nu^{-1}$-Lipschitz and completes the proof.
\end{proof}

With \Cref{lem:global-gradient-geometry} established, we return to the spectral range in \eqref{eq:smooth-class-definition}. Its inner and outer intervals have centers and half-widths
\begin{equation*}
c_\mu:=\frac{\mu_+-\mu_-}{2},\quad \bar{\mu}:=\frac{\mu_++\mu_-}{2},
\qquad
c_L:=\frac{L_+-L_-}{2},\quad \bar{L}:=\frac{L_++L_-}{2}.
\end{equation*}

Applying part~\ref{item:gradient-lipschitz} of \Cref{lem:global-gradient-geometry} to $E-\frac{c_L}{2}\norm{\cdot}^2$ with $\ell=\bar{L}$, and part~\ref{item:gradient-co-lipschitz} to $E-\frac{c_\mu}{2}\norm{\cdot}^2$ with $\nu=\bar{\mu}$, gives immediately the following exact secant characterization of $\classF_{\mu_\pm,L_\pm}$.

\begin{corollary}[Exact secant characterization]
\label[corollary]{cor:secant-geometry}
Let $E\in C^2(\R^d)$, and for $x_1,x_2\in\R^d$ write $p_i:=g(x_i)$, $\Delta x:=x_1-x_2$, and $\Delta p:=p_1-p_2$. Then $E\in\classF_{\mu_\pm,L_\pm}(\R^d)$ if and only if the following inequalities hold for every $x_1,x_2\in\R^d$:
\begin{subequations}\label{eq:centered-secant}
\begin{align}
\norm{\Delta p-c_\mu\Delta x}&\ge\bar{\mu}\norm{\Delta x},\\
\norm{\Delta p-c_L\Delta x}&\le\bar{L}\norm{\Delta x}.
\label{eq:centered-secant-upper}
\end{align}
\end{subequations}
Equivalently, squaring and expanding \eqref{eq:centered-secant} gives
\begin{subequations}\label{eq:exact-secant}
\begin{align}
\norm{\Delta p}^2&\ge(\mu_+-\mu_-)\langle\Delta p,\Delta x\rangle+\mu_+\mu_-\norm{\Delta x}^2,
\label{eq:exact-secant-lower}\\
\norm{\Delta p}^2&\le(L_+-L_-)\langle\Delta p,\Delta x\rangle+L_+L_-\norm{\Delta x}^2.
\label{eq:exact-secant-upper}
\end{align}
\end{subequations}
\end{corollary}

The two secant estimates play different roles. The upper estimate \eqref{eq:exact-secant-upper} is the familiar smoothness control. By contrast, the lower estimate \eqref{eq:exact-secant-lower} supplies inverse metric control even though the one-sided order furnished by strong monotonicity is unavailable in the indefinite components. This metric rigidity is developed next and later becomes the principal geometric input to paired proximal descent.

The corollary also shows that the Hessian is not needed to state the class, so from this point onward we retain $\classF_{\mu_\pm,L_\pm}$ for the $C^1$ class
\begin{equation*}
\classF_{\mu_\pm,L_\pm}(\R^d):=\bigl\{E\in C^1(\R^d):\ \eqref{eq:exact-secant}\text{ holds for every }x_1,x_2\in\R^d\bigr\}.
\end{equation*}
This convention enlarges the class introduced in \eqref{eq:smooth-class-definition}; by \Cref{cor:secant-geometry}, its $C^2$ members are exactly the functions satisfying the original spectral condition. From this point onward, the assumptions are therefore entirely first order.

The centered form \eqref{eq:centered-secant} immediately gives the bi-Lipschitz estimates
\begin{equation}
\mu_{\min}\norm{\Delta x}
\le\norm{\Delta p}
\le L_{\max}\norm{\Delta x},
\label{eq:ordinary-bilip}
\end{equation}
where
\begin{equation*}
\mu_{\min}:=\min\{\mu_+,\mu_-\},
\qquad
L_{\max}:=\max\{L_+,L_-\}.
\end{equation*}
Indeed, the reverse and direct triangle inequalities applied to \eqref{eq:centered-secant} give the two bounds. Hence $g$ is injective and has a Lipschitz inverse on its image. Invariance of domain makes this image open, while the lower estimate makes it closed; consequently $g$ is a global bi-Lipschitz homeomorphism of $\R^d$ and has a unique zero $x^\star$. This global correspondence is the metric rigidity supplied by uniform nondegeneracy.

Beyond guaranteeing uniqueness, the defining lower secant inequality \eqref{eq:exact-secant-lower} localizes the stationary point explicitly in terms of $x$ and $g(x)$. Indeed, applying it to the pair $(x,x^\star)$ gives the following certificate.

\begin{corollary}[A priori stationary-point localization]
\label[corollary]{cor:root-certification}
Let $E\in\classF_{\mu_\pm,L_\pm}(\R^d)$, and let $x^\star$ be its unique stationary point. For every $x\in\R^d$,
\begin{equation}
x^\star\in\Bcert(x):=\overline{B}\left(x+\frac{\mu_+-\mu_-}{2\mu_+\mu_-}g(x),\,\frac{\mu_++\mu_-}{2\mu_+\mu_-}\norm{g(x)}\right),
\label{eq:root-certification-ball}
\end{equation}
where $\overline{B}(z,r)$ denotes the closed Euclidean ball with center $z$ and radius $r$. We call $\Bcert(x)$ the certification ball at $x$.
\end{corollary}

\begin{proof}
Set $\Delta x:=x-x^\star$. Since $g(x^\star)=0$, \eqref{eq:exact-secant-lower} gives
\begin{equation*}
\mu_+\mu_-\norm{\Delta x}^2+(\mu_+-\mu_-)\langle g(x),\Delta x\rangle\le\norm{g(x)}^2.
\end{equation*}
Dividing by $\mu_+\mu_-$ and completing the square yields
\begin{equation*}
\norm{\Delta x+\frac{\mu_+-\mu_-}{2\mu_+\mu_-}g(x)}^2
\le\frac{(\mu_++\mu_-)^2}{4\mu_+^2\mu_-^2}\norm{g(x)}^2,
\end{equation*}
which is precisely \eqref{eq:root-certification-ball}.
\end{proof}

Thus $\Bcert(x)$ localizes the stationary point using only $x$, $g(x)$, and the inner endpoints $\mu_\pm$, making explicit the metric rigidity encoded by the lower secant estimate \eqref{eq:exact-secant-lower}. In particular, $x^\star\in\overline{B}(x,\norm{g(x)}/\mu_{\min})$.

The global homeomorphism $g$ also permits a Legendre transformation in the classical hyperregular form \cite[Secs.~7.2 and 7.4]{MarsdenRatiu1999}. For $E\in\classF_{\mu_\pm,L_\pm}(\R^d)$, define $E^\sharp:\R^d\to\R$ by
\begin{equation}
E^\sharp(p):=\langle p,g^{-1}(p)\rangle-E(g^{-1}(p)).
\label{eq:legendre-dual}
\end{equation}
The following result shows that this transformation preserves the intrinsic secant class and exchanges its spectral endpoints reciprocally.

\begin{proposition}[Reciprocal Legendre duality]
\label[proposition]{prop:legendre-duality}
The Legendre transform $E^\sharp$ is continuously differentiable with gradient $\nabla E^\sharp(p)=g^{-1}(p)$, and satisfies the Legendre identity
\begin{equation}
E(x)+E^\sharp(g(x))=\langle g(x),x\rangle.
\label{eq:legendre-identity}
\end{equation}
Moreover, $E^\sharp\in\classF_{\mu_\pm^\sharp,L_\pm^\sharp}(\R^d)$ with reciprocal endpoints $\mu_\pm^\sharp:=L_\pm^{-1}$ and $L_\pm^\sharp:=\mu_\pm^{-1}$, and the transformation is involutive: $(E^\sharp)^\sharp=E$.
\end{proposition}

\begin{proof}
The derivative formula follows from the global Lipschitz inverse. Fix $p\in\R^d$, write $x:=g^{-1}(p)$ and $x_h:=g^{-1}(p+h)$, and note from \eqref{eq:ordinary-bilip} that $x_h-x=O(\norm{h})$. The differentiability of $E$ at $x$ gives
\begin{equation*}
E^\sharp(p+h)-E^\sharp(p)=\langle p,x_h-x\rangle+\langle h,x_h\rangle-[E(x_h)-E(x)]=\langle h,x\rangle+o(\norm{h}),
\end{equation*}
so $\nabla E^\sharp(p)=x=g^{-1}(p)$. The identity \eqref{eq:legendre-identity} follows immediately from the definition.

To identify the secant geometry of the dual, take arbitrary $p_1,p_2\in\R^d$, let $x_i:=g^{-1}(p_i)$, and write $\Delta p:=p_1-p_2$ and $\Delta x:=x_1-x_2$. Since $\Delta p=g(x_1)-g(x_2)$, rearranging the upper and lower inequalities in \eqref{eq:exact-secant}, respectively, gives
\begin{align*}
\norm{\Delta x}^2
&\ge(L_+^{-1}-L_-^{-1})\langle\Delta x,\Delta p\rangle
+(L_+L_-)^{-1}\norm{\Delta p}^2,\\
\norm{\Delta x}^2
&\le(\mu_+^{-1}-\mu_-^{-1})\langle\Delta x,\Delta p\rangle
+(\mu_+\mu_-)^{-1}\norm{\Delta p}^2.
\end{align*}
Because $\Delta x=\nabla E^\sharp(p_1)-\nabla E^\sharp(p_2)$, these are exactly the defining secant inequalities for the reciprocal endpoints $\mu_\pm^\sharp=L_\pm^{-1}$ and $L_\pm^\sharp=\mu_\pm^{-1}$. Finally, the inverse of $\nabla E^\sharp=g^{-1}$ is $g$, and \eqref{eq:legendre-identity} gives $(E^\sharp)^\sharp(x)=\langle x,g(x)\rangle-E^\sharp(g(x))=E(x)$.
\end{proof}

Consequently, uniform nondegeneracy equips $\R^d$ with reciprocal primal and dual coordinates:
\begin{equation*}
x\in\R^d
\quad\xleftrightarrow[\;x=g^{-1}(p)=\nabla E^\sharp(p)\;]{\;p=g(x)=\nabla E(x)\;}
\quad p\in\R^d.
\end{equation*}

Under this reciprocal viewpoint, \Cref{cor:root-certification} takes the concise form
\begin{equation}
\Bcert(x)=\overline{B}\left(x-c_L^\sharp p,\,\bar{L}^\sharp\norm{p}\right),
\label{eq:reciprocal-certification-ball}
\end{equation}
where $p=g(x)$, and
\begin{equation*}
c_L^\sharp:=\frac{L_+^\sharp-L_-^\sharp}{2},
\qquad
\bar{L}^\sharp:=\frac{L_+^\sharp+L_-^\sharp}{2}
\end{equation*}
are the center and half-width of the reciprocal outer interval.

\begin{remark}[Relation to classical conjugacy]
On the convex component, \eqref{eq:legendre-dual} agrees with the usual Legendre--Fenchel conjugate \cite[Chap.~13]{BauschkeCombettes2017}. For an SCSC objective written in known positive--negative coordinates $x=(u_+,u_-)$, it likewise recovers, up to the conventional orientation, Rockafellar's conjugate saddle-function, whose saddle representation is $\sup_{u_+}\inf_{u_-}\{\langle p_+,u_+\rangle+\langle p_-,u_-\rangle-E(u_+,u_-)\}$ \cite{Rockafellar1964}.

Beyond these familiar subclasses, the same Legendre transform remains well defined without a fixed curvature splitting. For a general indefinite member, $g^{-1}(p)$ is the unique stationary point of $x\mapsto\langle p,x\rangle-E(x)$ but need not be an extremizer. Accordingly, $E^\sharp$ denotes the hyperregular Legendre transform, while $E^*$ is reserved for the ordinary Fenchel conjugate.
\end{remark}

In the dual coordinate $p=g(x)$, stationarity corresponds to the origin $p=0$, and $\norm{g(x)}=\norm{p}$ is precisely the distance to that origin. The reciprocal endpoint geometry further controls deviations of $E(x)$ both above and below the stationary value $E^\star:=E(x^\star)$, yielding a sign-resolved analogue of the PL inequality.

\begin{proposition}[Signed PL inequality]
\label[proposition]{prop:signed-pl}
Let $E\in\classF_{\mu_\pm,L_\pm}(\R^d)$. Then
\begin{equation}
\frac12\norm{g(x)}^2
\ge
\mu_+\bigl[E(x)-E^\star\bigr]_+
+\mu_-\bigl[E(x)-E^\star\bigr]_-,
\qquad \forall x\in\R^d.
\label{eq:signed-pl}
\end{equation}
Here, $[a]_+:=\max\{a,0\}$ and $[a]_-:=\max\{-a,0\}$.
\end{proposition}

\begin{proof}
By \Cref{prop:legendre-duality}, $\nabla E^\sharp=g^{-1}$ and the outer endpoints of $E^\sharp$ are $L_\pm^\sharp=\mu_\pm^{-1}$. Applying the upper centered secant estimate \eqref{eq:centered-secant-upper} to $E^\sharp$ and using Cauchy--Schwarz immediately gives, for every $p_1,p_2\in\R^d$,
\begin{equation}
-\frac1{\mu_-}\norm{p_1-p_2}^2
\le\langle g^{-1}(p_1)-g^{-1}(p_2),p_1-p_2\rangle
\le\frac1{\mu_+}\norm{p_1-p_2}^2.
\label{eq:inverse-sector}
\end{equation}
Set $p:=g(x)$. The Legendre identity \eqref{eq:legendre-identity} gives $E(x)=\langle g^{-1}(p),p\rangle-E^\sharp(p)$, while $E^\sharp(0)=-E^\star$. Since $\nabla E^\sharp=g^{-1}$, the fundamental theorem of calculus yields
\begin{equation*}
E(x)-E^\star
=\langle g^{-1}(p),p\rangle-[E^\sharp(p)-E^\sharp(0)]
=\int_0^1\langle g^{-1}(p)-g^{-1}(tp),p\rangle\dd t.
\end{equation*}
For $0\le t<1$, applying \eqref{eq:inverse-sector} to $p$ and $tp$ and dividing by $1-t$ gives
\begin{equation*}
-\frac{1-t}{\mu_-}\norm{p}^2
\le\langle g^{-1}(p)-g^{-1}(tp),p\rangle
\le\frac{1-t}{\mu_+}\norm{p}^2.
\end{equation*}
Integration over $t\in[0,1]$ gives
\begin{equation*}
-\frac1{2\mu_-}\norm{g(x)}^2
\le E(x)-E^\star
\le\frac1{2\mu_+}\norm{g(x)}^2,
\end{equation*}
which is equivalent to \eqref{eq:signed-pl}.
\end{proof}

For a strongly convex objective, the negative part vanishes and \eqref{eq:signed-pl} reduces to the classical PL inequality; after reversing the sign, the strongly concave case follows analogously. In the indefinite components, the classical estimate re-emerges in the signed form \eqref{eq:signed-pl}, simultaneously controlling energy deviations above and below $E^\star$.

\subsection{Index structure and classical subclasses}

The secant characterization \eqref{eq:exact-secant} gives a precise first-order description of $\classF_{\mu_\pm,L_\pm}$, but the resulting function class may still appear abstract. We now make its structure more concrete by identifying its intrinsic curvature types and locating familiar optimization models among them. 

We start with the familiar $C^2$ case. Because the spectral gap excludes zero eigenvalues, the Hessian $H(x^\star)$ has $k$ negative and $d-k$ positive eigenvalues for some $k\in\{0,\ldots,d\}$. This integer is the Morse index of the stationary point \cite[Sec.~2]{Milnor1963}: $k=0$ describes a minimum, $k=d$ a maximum, and $1\le k\le d-1$ a saddle with $k$ negative-curvature directions.

For a $C^1$ objective, the Hessian need not exist at $x^\star$, so we instead use the Conley index of the associated gradient flow $\dot x=-g(x)$ \cite[pp.~43--52]{Conley1978}. Within the secant class, this pointed homotopy type is $\mathbb{S}^k$ for a unique $k\in\{0,\ldots,d\}$, and we denote the corresponding subclass by $\classF_{\mu_\pm,L_\pm}^{[k]}(\R^d)$. In the $C^1_{\mathrm{loc}}$ topology, these $d+1$ curvature types are precisely the connected components of the secant class:
\begin{equation}
\classF_{\mu_\pm,L_\pm}(\R^d)=\bigsqcup_{k=0}^d\classF_{\mu_\pm,L_\pm}^{[k]}(\R^d).
\label{eq:index-decomposition}
\end{equation}
The precise statement and topological proof are given in \Cref{app:index-components}. We next locate familiar optimization models within this decomposition.

\begin{proposition}[Definite and SCSC subclasses]
\label[proposition]{prop:classical-subclasses}
\begin{enumerate}
\item\label{item:definite-subclasses} Index zero is exactly the $\mu_+$-strongly convex and $L_+$-smooth function class. Index $d$ is exactly the $\mu_-$-strongly concave and $L_-$-smooth function class.

\item\label{item:minimax-subclass} Let $d_x,d_y\ge1$, $0<\mu_x\le L_x$, $0<\mu_y\le L_y$, and $L_{xy}\ge0$, and let $E\in C^1(\R^{d_x}\times\R^{d_y})$ be $\mu_x$-strongly convex in $x$ and $\mu_y$-strongly concave in $y$. Assume that
\begin{align*}
\norm{\nabla_xE(x_1,y_1)-\nabla_xE(x_2,y_2)}&\le L_x\norm{x_1-x_2}+L_{xy}\norm{y_1-y_2},\\
\norm{\nabla_yE(x_1,y_1)-\nabla_yE(x_2,y_2)}&\le L_{xy}\norm{x_1-x_2}+L_y\norm{y_1-y_2}.
\end{align*}
Then $E\in\classF_{\mu_\pm,L_\pm}^{[d_y]}(\R^{d_x+d_y})$, where $\mu_+:=\mu_x$, $\mu_-:=\mu_y$, and
\begin{subequations}\label{eq:scsc-outer-endpoints}
\begin{align}
L_+&:=\frac{L_x-\mu_y+\sqrt{(L_x+\mu_y)^2+4L_{xy}^2}}2,\\
L_-&:=\frac{L_y-\mu_x+\sqrt{(L_y+\mu_x)^2+4L_{xy}^2}}2.
\end{align}
\end{subequations}
\end{enumerate}
\end{proposition}

\begin{proof}
Both parts use the same elementary smoothing fact: mollification preserves strong convexity, strong concavity, and SCSC structure, with all stated constants unchanged. Specifically, for an objective $E$ in any of these classes, set $E_\epsilon:=E*\varphi_\epsilon$, where $\varphi_\epsilon$ is a standard compactly supported mollifier. Then $E_\epsilon$ is smooth, remains in the same class with the same constants, and converges to $E$ in $C^1_{\mathrm{loc}}$. It therefore suffices to establish the asserted spectral bounds and inertia for $E_\epsilon$; \Cref{cor:secant-geometry} passes the resulting secant inequalities to the limit, while the mollification path preserves the component index.

For part~\ref{item:definite-subclasses}, the smooth spectral claim is immediate: strong convexity gives the positive interval $[\mu_+,L_+]$, while strong concavity gives the negative interval $[-L_-,-\mu_-]$. Conversely, \Cref{prop:index-components} places the weak Hessian of a member of the zero- or $d$-index component in the corresponding interval almost everywhere, yielding the stated strongly convex or concave class.

For part~\ref{item:minimax-subclass}, the mollified Hessian has the form
\begin{equation*}
H_\epsilon=\begin{pmatrix}A&B\\B^\top&-C\end{pmatrix},
\qquad
\mu_xI\preceq A\preceq L_xI,\quad
\mu_yI\preceq C\preceq L_yI,\quad
\norm{B}\le L_{xy}.
\end{equation*}
The block calculation in \Cref{app:classical-subclasses} gives $\spec(H_\epsilon)\subseteq[-L_-,-\mu_y]\cup[\mu_x,L_+]$ and shows that $H_\epsilon$ has exactly $d_y$ negative eigenvalues. Thus every $E_\epsilon$ belongs to the asserted smooth spectral class with index $d_y$, and the common limiting argument gives $E\in\classF_{\mu_\pm,L_\pm}^{[d_y]}(\R^{d_x+d_y})$.
\end{proof}

More importantly, the intrinsic secant class extends strictly beyond these familiar models: it contains genuinely indefinite objectives with no fixed convex--concave splitting. The following examples make this enlargement concrete: the first gives an oscillatory saddle with no fixed orthogonal splitting, while the second lets the negative-curvature subspace itself rotate indefinitely.

\begin{example}[Oscillation without a fixed orthogonal splitting]
\label[example]{ex:oscillatory-saddle}
Fix $b>1$ and $K>0$, and define $E(x,y):=xy-2bK^{-2}\cos(Kx)$ on $\R^2$. Its gradient and Hessian are
\begin{equation*}
g(x,y)=\begin{pmatrix}y+(2b/K)\sin(Kx)\\x\end{pmatrix},
\qquad
\nabla^2E(x,y)=\begin{pmatrix}2b\cos(Kx)&1\\1&0\end{pmatrix}.
\end{equation*}
Writing $a:=2b\cos(Kx)\in[-2b,2b]$, the characteristic polynomial of the Hessian is $\lambda^2-a\lambda-1$. Hence, with $L:=\sqrt{b^2+1}+b$ and $\mu:=L^{-1}=\sqrt{b^2+1}-b$, its eigenvalues remain in $[-L,-\mu]\cup[\mu,L]$. Solving $g(x,y)=0$ gives the origin, so $E\in\classF_{\mu,L}^{[1]}(\R^2)$ has a unique saddle at $(0,0)$.

For this objective, no fixed orthogonal convex--concave splitting exists. Suppose otherwise, and let $e=(e_1,e_2)$ be the uniformly positive direction and $e^\perp=(-e_2,e_1)$ the uniformly negative direction. At $a=-2b$, positivity along $e$ gives $e_1e_2>be_1^2$, whereas at $a=2b$, negativity along $e^\perp$ gives $e_1e_2>be_2^2$. These inequalities force $e_1e_2>0$; multiplying them and cancelling $e_1^2e_2^2$ yields $1>b^2$, contradicting $b>1$.

This example also separates spectral conditioning from Hessian Lipschitz regularity. The spectral endpoints depend only on $b$, whereas $\Lip(\nabla^2E)=2bK$. Increasing $K$ therefore makes the Hessian vary arbitrarily rapidly without leaving the same uniformly nondegenerate class.
\end{example}

\begin{example}[Rotating curvature subspaces]
\label[example]{ex:rotating-curvature}
We next construct an objective whose curvature subspaces rotate indefinitely while its spectral gap remains uniform. Let $A\in C^2([0,\infty);\R^{d\times d})$ be symmetric-matrix-valued, with every $A(r)$ having spectrum in $[-L,-\mu]\cup[\mu,L]$ and exactly $k$ negative eigenvalues. Define the objective and its radial variation by
\begin{equation*}
E(x):=\frac12x^\top A(\norm{x})x,
\qquad
\delta_A:=\frac52\sup_{r\ge0}r\norm{A'(r)}+\frac12\sup_{r\ge0}r^2\norm{A''(r)}.
\end{equation*}
For $r:=\norm{x}>0$, set $e_r:=x/r$ and $p_r:=A'(r)x$. Direct differentiation gives
\begin{equation}
\nabla^2E(x)=A(r)+p_re_r^\top+e_rp_r^\top+\frac12\langle p_r,e_r\rangle(I-e_re_r^\top)+\frac12[x^\top A''(r)x]e_re_r^\top.
\label{eq:radial-hessian}
\end{equation}
The triangle inequality yields
\begin{equation*}
\norm{\nabla^2E(x)-A(r)}
\le\frac52\norm{p_r}+\frac12r^2\norm{A''(r)}
\le\delta_A,
\end{equation*}
and the same estimate holds at $x=0$ by continuity. Weyl's inequality therefore gives $E\in\classF_{\mu-\delta_A,L+\delta_A}^{[k]}(\R^d)$ whenever $\delta_A<\mu$.

To obtain unbounded rotation, specialize to $d=2$ and set
\begin{equation*}
A(r):=\Rot(\vartheta(r))\diag(1,-1)\Rot(\vartheta(r))^\top,
\qquad
\vartheta(r):=\alpha\log(1+r^2),
\end{equation*}
where $\Rot$ denotes planar rotation. Differentiating gives $\norm{A'(r)}=2|\vartheta'(r)|$ and $\norm{A''(r)}\le2|\vartheta''(r)|+4|\vartheta'(r)|^2$, so $\delta_A=O(\alpha+\alpha^2)<1$ for sufficiently small $\alpha>0$.

Thus the objective remains uniformly nondegenerate even though $\vartheta(r)\to\infty$. As the reference eigenspaces pass through every orientation and \eqref{eq:radial-hessian} remains a perturbation of size less than one, every fixed nonzero direction encounters both signs of curvature. Hence no fixed curvature splitting exists.
\end{example}

Together, the examples show that the enlargement beyond convex and SCSC models is structural, not merely a loss of regularity: uniform nondegeneracy controls global first-order geometry without imposing either a common positive--negative decomposition or a modulus of Hessian continuity. The secant estimates, rather than a fixed splitting, are therefore the natural data from which to construct a first-order method.

\section{First-order computation on the whole space}
\label{sec:proximal}

With the first-order geometry of \Cref{sec:class} in place, we now ask whether it can be converted into a first-order method for computing the unique stationary point. The metric estimate \eqref{eq:ordinary-bilip} makes $\norm{g(x)}$ both an observable measure of stationarity and a quantitative proxy for $\norm{x-x^\star}$. This suggests minimizing the squared residual $\mathcal{R}(x):=\norm{g(x)}^2$. If $E\in C^2(\R^d)$, then $\nabla\mathcal{R}(x)=2H(x)g(x)$, and uniform nondegeneracy gives the PL inequality \cite{Polyak1963,KarimiNutiniSchmidt2016}
\begin{equation*}
\frac12\norm{\nabla\mathcal{R}(x)}^2\ge2\mu_{\min}^2\mathcal{R}(x).
\end{equation*}
Thus $\mathcal{R}$ appears to supply the missing descent geometry.

However, its descent direction lies outside the intended first-order model. Under the intrinsic $C^1$ formulation, the Hessian $H$ need not exist everywhere and $\mathcal{R}$ is only locally Lipschitz; even under $C^2$ regularity, the spectral endpoints $\mu_\pm,L_\pm$ do not control the variation of $H$, and hence do not guarantee a Lipschitz gradient for $\mathcal{R}$. We therefore seek a smooth surrogate for $\norm{g}^2$ that preserves its control of stationarity while admitting a gradient realizable from first-order queries.

\subsection{Paired proximal geometry}

We begin with a simple observation: the classical Moreau construction \cite[Sec.~12.4]{BauschkeCombettes2017} admits a natural signed extension to the present two-sided spectrum.

\begin{lemma}[Two-sided Moreau family]
\label[lemma]{lem:proximal-path}
For any $x\in\R^d$ and $s\in(-L_-^{-1},L_+^{-1})$, the function
\begin{equation*}
f_{s,x}(y):=s\bigl(E(y)-E^\star\bigr)-\frac12\norm{y-x}^2
\end{equation*}
is strongly concave. Denote its maximum value and unique maximizer by
\begin{equation*}
\mathcal{M}_s(x):=\max_{y\in\R^d}f_{s,x}(y),
\qquad
y_s(x):=\operatorname*{argmax}_{y\in\R^d}f_{s,x}(y).
\end{equation*}
Then $\mathcal{M}_s\in C^1(\R^d)$ with Lipschitz-continuous gradient $\nabla\mathcal{M}_s(x)=y_s(x)-x$.
\end{lemma}

\begin{proof}
The upper secant inequality \eqref{eq:exact-secant-upper}, together with Cauchy--Schwarz, gives
\begin{equation*}
-L_-\norm{y_1-y_2}^2
\le\langle g(y_1)-g(y_2),y_1-y_2\rangle
\le L_+\norm{y_1-y_2}^2.
\end{equation*}
Since $\nabla_y f_{s,x}(y)=sg(y)-(y-x)$, these bounds give
\begin{equation*}
\left\langle \nabla_y f_{s,x}(y_1)-\nabla_y f_{s,x}(y_2),y_1-y_2\right\rangle
\le-m_s\norm{y_1-y_2}^2,
\end{equation*}
where
\begin{equation}
m_s:=
\begin{cases}
1-sL_+,&s\ge0,\\
1+sL_-,&s<0.
\end{cases}
\label{eq:proximal-path-modulus}
\end{equation}
The parameter range $s\in(-L_-^{-1},L_+^{-1})$ ensures $m_s>0$. Hence $f_{s,x}$ is $m_s$-strongly concave, its maximum is uniquely attained, and the optimality condition for the maximizer $y_s(x)$ is
\begin{equation}
sg(y_s(x))=y_s(x)-x.
\label{eq:proximal-path-optimality}
\end{equation}
Thus $y_s$ is the inverse of the $m_s$-strongly monotone map $\Id-sg$. Applying strong monotonicity at $y_s(x_1)$ and $y_s(x_2)$ gives
\begin{align*}
m_s\norm{y_s(x_1)-y_s(x_2)}^2
&\le\left\langle x_1-x_2,y_s(x_1)-y_s(x_2)\right\rangle\\
&\le\norm{x_1-x_2}\norm{y_s(x_1)-y_s(x_2)},
\end{align*}
and hence $y_s$ is Lipschitz continuous with $\Lip(y_s)\le m_s^{-1}$. Danskin's theorem \cite[Chap.~3]{Danskin1967} now gives $\nabla\mathcal{M}_s(x)=y_s(x)-x$, and thus $\mathcal{M}_s\in C^1(\R^d)$ with
\begin{equation*}
\Lip(\nabla\mathcal{M}_s)\le1+m_s^{-1}.
\end{equation*}
\end{proof}

In familiar proximal notation, $s\mapsto y_s(x)$ is a two-sided proximal path passing through the anchor $x$:
\begin{equation*}
y_s(x)=
\begin{cases}
\prox_{s(-E)}(x),&s>0,\\
x,&s=0,\\
\prox_{(-s)E}(x),&s<0.
\end{cases}
\end{equation*}
Thus the positive branch regularizes $-E$, while the negative branch regularizes $E$.

To construct the surrogate, we select one parameter value from each branch of this path. Choose penalty parameters $\Lambda_->L_-$ and $\Lambda_+>L_+$, so that $-1/\Lambda_-$ and $1/\Lambda_+$ lie in the admissible interval. Denote the corresponding envelopes and maximizers by
\begin{equation}
\begin{aligned}
\mathcal{M}_-&:=\mathcal{M}_{-1/\Lambda_-},
&\mathcal{M}_+&:=\mathcal{M}_{1/\Lambda_+},\\
y_-&:=y_{-1/\Lambda_-},
&y_+&:=y_{1/\Lambda_+}.
\end{aligned}
\label{eq:proximal-endpoints}
\end{equation}
At these two endpoints, the optimality condition \eqref{eq:proximal-path-optimality} becomes
\begin{equation*}
y_-(x)=x-\frac1{\Lambda_-}g\bigl(y_-(x)\bigr),
\qquad
y_+(x)=x+\frac1{\Lambda_+}g\bigl(y_+(x)\bigr).
\end{equation*}
Thus $y_-(x)$ and $y_+(x)$ are, respectively, an implicit gradient-descent step and an implicit gradient-ascent step from $x$, with stepsizes $1/\Lambda_-$ and $1/\Lambda_+$. Neither branch alone supplies global descent in the intermediate components. The central insight is that, when appropriately weighted and paired, the two branches produce a global descent merit that neither possesses individually.

For $c\in(-\mu_-,\mu_+)$, define the normalized distances
\begin{equation}
w_-:=\frac{\Lambda_-+c}{\Lambda_++\Lambda_-},
\qquad
w_+:=\frac{\Lambda_+-c}{\Lambda_++\Lambda_-},
\label{eq:proximal-weights}
\end{equation}
and introduce the paired proximal merit $V$ as the weighted average
\begin{equation}
V(x):=w_-\mathcal{M}_-(x)+w_+\mathcal{M}_+(x).
\label{eq:proximal-merit}
\end{equation}
The parameter geometry behind this average is summarized in \Cref{fig:paired-merit-geometry}.

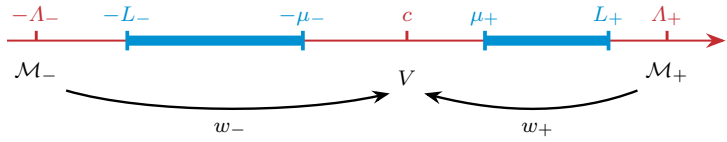
\begin{figure}[ht]
\centering
\begin{tikzpicture}[x=0.86cm,y=0.72cm]
\definecolor{spectralblue}{RGB}{0,145,210}
\definecolor{parameterred}{RGB}{195,45,50}
\draw[parameterred,-{Stealth[length=2.6mm]},line width=0.75pt] (0.2,1.45)--(11.25,1.45);
\draw[spectralblue,line width=3.8pt] (2.05,1.45)--(4.75,1.45);
\draw[spectralblue,line width=3.8pt] (7.55,1.45)--(9.45,1.45);
\foreach \x in {0.65,6.35,10.35}
  \draw[parameterred,line width=1.1pt] (\x,1.45)--(\x,1.63);
\foreach \x in {2.05,4.75,7.55,9.45}
  \draw[spectralblue,line width=1.4pt] (\x,1.27)--(\x,1.63);
\node[parameterred,anchor=base] at (0.65,1.82) {$-\Lambda_-$};
\node[spectralblue,anchor=base] at (2.05,1.82) {$-L_-$};
\node[spectralblue,anchor=base] at (4.75,1.82) {$-\mu_-$};
\node[parameterred,anchor=base] at (6.35,1.82) {$c$};
\node[spectralblue,anchor=base] at (7.55,1.82) {$\mu_+$};
\node[spectralblue,anchor=base] at (9.45,1.82) {$L_+$};
\node[parameterred,anchor=base] at (10.35,1.82) {$\Lambda_+$};
\node (Mminus) at (0.65,0.86) {$\mathcal{M}_-$};
\node (Vnode) at (6.35,0.76) {$V$};
\node (Mplus) at (10.35,0.86) {$\mathcal{M}_+$};
\draw[-{Stealth[length=2.3mm]},line width=0.9pt]
  (Mminus.south east) .. controls (2.35,0.10) and (4.65,0.10) ..
  node[pos=0.52,below=2pt] {$w_-$} (Vnode.south west);
\draw[-{Stealth[length=2.3mm]},line width=0.9pt]
  (Mplus.south west) .. controls (8.95,0.10) and (7.85,0.10) ..
  node[pos=0.50,below=2pt] {$w_+$} (Vnode.south east);
\end{tikzpicture}
\caption{Parameter geometry of the paired proximal merit. The normalized distances from $c$ to the outer penalties $\Lambda_\pm$ determine the weights $w_\pm$.}
\label{fig:paired-merit-geometry}
\vspace{-8pt}
\end{figure}

Define $v_\pm(x):=\nabla\mathcal{M}_\pm(x)=y_\pm(x)-x$. By \Cref{lem:proximal-path}, $V\in C^1(\R^d)$ with Lipschitz-continuous gradient
\begin{align}
\nabla V(x)&=v(x):=w_-v_-(x)+w_+v_+(x),
\label{eq:paired-merit-gradient}\\
\Lip(v)&\le L_v:=1+w_-/m_-+w_+/m_+,
\label{eq:merit-gradient-lipschitz}
\end{align}
where $m_-:=1-L_-/\Lambda_-$ and $m_+:=1-L_+/\Lambda_+$. Equivalently, $V$ is $L_v$-smooth.

Beyond smoothness, the decisive result is that the merit $V$ recovers the missing descent geometry of $\mathcal{R}$ at both zeroth and first order:
\begin{equation*}
V(x)\asymp\norm{g(x)}^2,
\qquad
\norm{v(x)}\asymp\norm{g(x)},
\end{equation*}
where $\asymp$ denotes uniform two-sided bounds up to positive multiplicative constants.

\begin{theorem}[Merit equivalence]
\label[theorem]{thm:paired-proximal-equivalence}
There exist constants $\overline{C}_0\ge\underline{C}_0>0$ and $\overline{C}_1\ge\underline{C}_1>0$, depending only on $\mu_\pm,L_\pm,\Lambda_\pm,c$, such that, for every $x\in\R^d$,
\begin{subequations}
\begin{align}
&\underline{C}_0\norm{g(x)}^2
\le V(x)
\le\overline{C}_0\norm{g(x)}^2,
\label{eq:merit-value-equivalence}\\
&\underline{C}_1\norm{g(x)}
\le\norm{v(x)}
\le\overline{C}_1\norm{g(x)}.
\label{eq:merit-gradient-equivalence}
\end{align}
\end{subequations}
\end{theorem}

\begin{proof}
Fix $x\in\R^d$ and abbreviate $y_s:=y_s(x)$. We first show that $\norm{g(y_s)}\asymp\norm{g(x)}$ along the proximal path. Apply the upper secant estimate \eqref{eq:centered-secant-upper} to $x$ and $y_s$, and use the optimality condition \eqref{eq:proximal-path-optimality} to obtain
\begin{equation*}
\norm{g(x)-(1-sc_L)g(y_s)}
\le |s|\bar{L}\norm{g(y_s)}.
\end{equation*}
It follows that
\begin{equation}
\min\{1-sL_+,1+sL_-\}\norm{g(y_s)}
\le\norm{g(x)}
\le\max\{1-sL_+,1+sL_-\}\norm{g(y_s)}.
\label{eq:proximal-path-residual}
\end{equation}

We now use this estimate to control the value of $V$. Danskin's theorem, applied with $s$ as the parameter, gives
\begin{equation*}
\partial_s\mathcal{M}_s(x)=E(y_s)-E^\star.
\end{equation*}
Substituting \eqref{eq:proximal-path-optimality} into the definition of $\mathcal{M}_s$ also gives
\begin{equation*}
\mathcal{M}_s(x)
=f_{s,x}(y_s)
=s\bigl(E(y_s)-E^\star\bigr)-\frac{s^2}{2}\norm{g(y_s)}^2.
\end{equation*}
Hence, for $s\ne0$, the product rule gives
\begin{align*}
\partial_s\left[\left(\frac1s-c\right)\mathcal{M}_s(x)\right]
&=-\frac{\mathcal{M}_s(x)}{s^2}
+\left(\frac1s-c\right)\bigl(E(y_s)-E^\star\bigr)\\
&=\frac12\norm{g(y_s)}^2-c\bigl(E(y_s)-E^\star\bigr).
\end{align*}
As $s\to0$, \eqref{eq:proximal-path-optimality} and \eqref{eq:proximal-path-residual} imply $y_s\to x$; hence the apparent singularity at $s=0$ is removable, and the differential identity extends continuously through zero. Integrating over $s\in[-1/\Lambda_-,1/\Lambda_+]$ and using the weights \eqref{eq:proximal-weights} gives
\begin{equation}
(\Lambda_++\Lambda_-)V(x)
=\int_{-1/\Lambda_-}^{1/\Lambda_+}
\left[
\frac12\norm{g(y_s)}^2-c\bigl(E(y_s)-E^\star\bigr)
\right]\dd s.
\label{eq:merit-integral-representation}
\end{equation}
\enlargethispage{\baselineskip}
Moreover, the signed PL inequality of \Cref{prop:signed-pl} gives
\begin{equation*}
\begin{aligned}
\frac12\min\left\{\frac{\widetilde{\mu}_+}{\mu_+},\frac{\widetilde{\mu}_-}{\mu_-}\right\}\norm{g(y_s)}^2
&\le\frac12\norm{g(y_s)}^2-c\bigl(E(y_s)-E^\star\bigr)\\
&\le\frac12\max\left\{\frac{\widetilde{\mu}_+}{\mu_+},\frac{\widetilde{\mu}_-}{\mu_-}\right\}\norm{g(y_s)}^2,
\end{aligned}
\end{equation*}
where $\widetilde{\mu}_+:=\mu_+-c$ and $\widetilde{\mu}_-:=\mu_-+c$ are the shifted inner endpoints.

Together with \eqref{eq:merit-integral-representation}, this shows that
\begin{equation*}
V(x)\asymp\frac1{\Lambda_++\Lambda_-}
\int_{-1/\Lambda_-}^{1/\Lambda_+}\norm{g(y_s)}^2\dd s.
\end{equation*}
Applying \eqref{eq:proximal-path-residual} separately on $s\ge0$ and $s\le0$, then squaring and integrating, gives
\begin{equation*}
\begin{aligned}
\left(\frac1{\Lambda_++L_-}+\frac1{\Lambda_-+L_+}\right)\norm{g(x)}^2
&\le\int_{-1/\Lambda_-}^{1/\Lambda_+}\norm{g(y_s)}^2\dd s\\
&\le\left(\frac1{\Lambda_+-L_+}+\frac1{\Lambda_--L_-}\right)\norm{g(x)}^2.
\end{aligned}
\end{equation*}
Therefore, the merit value equivalence \eqref{eq:merit-value-equivalence} holds with
\begin{align*}
\underline{C}_0
&:=\frac{1}{2(\Lambda_++\Lambda_-)}
\min\left\{\frac{\widetilde{\mu}_+}{\mu_+},\frac{\widetilde{\mu}_-}{\mu_-}\right\}
\left(\frac1{\Lambda_++L_-}+\frac1{\Lambda_-+L_+}\right),\\
\overline{C}_0
&:=\frac{1}{2(\Lambda_++\Lambda_-)}
\max\left\{\frac{\widetilde{\mu}_+}{\mu_+},\frac{\widetilde{\mu}_-}{\mu_-}\right\}
\left(\frac1{\Lambda_+-L_+}+\frac1{\Lambda_--L_-}\right).
\end{align*}

We next prove the merit gradient equivalence \eqref{eq:merit-gradient-equivalence}. Using $v_\pm(x)=y_\pm(x)-x$, the optimality condition \eqref{eq:proximal-path-optimality} at $s=1/\Lambda_+$ and $s=-1/\Lambda_-$ becomes
\begin{equation}
g(y_+)=\Lambda_+v_+,
\qquad
g(y_-)=-\Lambda_-v_-.
\label{eq:proximal-endpoint-optimality}
\end{equation}
Combining these with the residual comparison \eqref{eq:proximal-path-residual} gives
\begin{equation*}
\norm{v_+}=\frac{\norm{g(y_+)}}{\Lambda_+}
\le\frac{\norm{g(x)}}{\Lambda_+-L_+},
\qquad
\norm{v_-}=\frac{\norm{g(y_-)}}{\Lambda_-}
\le\frac{\norm{g(x)}}{\Lambda_--L_-}.
\end{equation*}
Substitution into \eqref{eq:paired-merit-gradient} proves the upper bound in \eqref{eq:merit-gradient-equivalence} with
\begin{equation*}
\overline{C}_1:=\frac1{\Lambda_++\Lambda_-}
\left(
\frac{\Lambda_+-c}{\Lambda_+-L_+}
+\frac{\Lambda_-+c}{\Lambda_--L_-}
\right).
\end{equation*}

For the lower bound in \eqref{eq:merit-gradient-equivalence}, the right-hand inequality in \eqref{eq:proximal-path-residual}, applied to $y_\pm$, gives
\begin{equation}
\begin{aligned}
\norm{g(x)}
&\le\left(1+\frac{L_-}{\Lambda_+}\right)\norm{g(y_+)}
=(\Lambda_++L_-)\norm{v_+},\\
\norm{g(x)}
&\le\left(1+\frac{L_+}{\Lambda_-}\right)\norm{g(y_-)}
=(\Lambda_-+L_+)\norm{v_-}.
\end{aligned}
\label{eq:residual-endpoint-control}
\end{equation}
It remains to control $v_\pm$ by their weighted average $v$. By the chosen weights $w_\pm$,
\begin{equation*}
(\Lambda_++\Lambda_-)v
=(\Lambda_+-c)v_++(\Lambda_-+c)v_-
=\Lambda_+v_++\Lambda_-v_--c(y_+-y_-).
\end{equation*}
Introducing the shifted gradient $\widetilde{g}:=g-c\,\Id$ and using the identities \eqref{eq:proximal-endpoint-optimality} gives
\begin{equation}
(\Lambda_++\Lambda_-)v=\widetilde{g}(y_+)-\widetilde{g}(y_-).
\label{eq:merit-gradient-cross-secant}
\end{equation}
Since $\widetilde{g}$ has lower secant modulus $\widetilde{\mu}_{\min}:=\min\{\widetilde{\mu}_+,\widetilde{\mu}_-\}$, \eqref{eq:merit-gradient-cross-secant} yields
\begin{equation}
(\Lambda_++\Lambda_-)\norm{v}
\ge\widetilde{\mu}_{\min}\norm{y_+-y_-}.
\label{eq:proximal-endpoint-separation}
\end{equation}
Moreover, \eqref{eq:paired-merit-gradient} gives
\begin{equation*}
v_+=v+w_-(y_+-y_-),
\qquad
v_-=v-w_+(y_+-y_-).
\end{equation*}
Using \eqref{eq:proximal-endpoint-separation} and the definitions of $w_\pm$ in \eqref{eq:proximal-weights}, these identities give
\begin{equation*}
\norm{v_+}
\le\left(1+\frac{\Lambda_-+c}{\widetilde{\mu}_{\min}}\right)\norm{v},\qquad
\norm{v_-}\le\left(1+\frac{\Lambda_+-c}{\widetilde{\mu}_{\min}}\right)\norm{v}.
\end{equation*}
Combining these displacement bounds with \eqref{eq:residual-endpoint-control} yields the lower constant
\begin{equation*}
\underline{C}_1:=\frac1{\min\left\{
(\Lambda_++L_-)\left(1+\dfrac{\Lambda_-+c}{\widetilde{\mu}_{\min}}\right),
(\Lambda_-+L_+)\left(1+\dfrac{\Lambda_+-c}{\widetilde{\mu}_{\min}}\right)
\right\}}.
\end{equation*}
This completes the proof.
\end{proof}

In particular, $V\ge0$ and $V(x)=0$ if and only if $x=x^\star$. Moreover, the upper value bound in \eqref{eq:merit-value-equivalence} and the lower gradient bound in \eqref{eq:merit-gradient-equivalence} give the global PL inequality
\begin{equation*}
\frac12\norm{\nabla V(x)}^2
\ge\frac{\underline{C}_1^2}{2\overline{C}_0}V(x).
\end{equation*}
Consequently, for $0<\tau<1$, standard smooth PL analysis \cite[Theorem~1]{KarimiNutiniSchmidt2016} gives linear convergence of
\begin{equation*}
x_{t+1}=x_t-\frac{\tau}{L_v}\nabla V(x_t)
=x_t-\frac{\tau}{L_v}v(x_t).
\end{equation*}
The remaining task is to realize this update using the first-order oracle.

\subsection{First-order implementation and query complexity}

At a given anchor $x$, evaluating the merit gradient $v(x)=w_+v_+(x)+w_-v_-(x)$ reduces to computing the two proximal points $y_\pm(x)$, which maximize the functions $f_{\pm,x}$ defined by
\begin{equation*}
f_{+,x}:=f_{1/\Lambda_+,x},
\qquad
f_{-,x}:=f_{-1/\Lambda_-,x}.
\end{equation*}
The same secant bounds used in \Cref{lem:proximal-path} show that $f_{\pm,x}$ is $m_\pm$-strongly concave and $M_\pm$-smooth, where
\begin{equation}
\begin{aligned}
m_+&:=1-\frac{L_+}{\Lambda_+},
&m_-&:=1-\frac{L_-}{\Lambda_-},\\
M_+&:=1+\frac{L_-}{\Lambda_+},
&M_-&:=1+\frac{L_+}{\Lambda_-}.
\end{aligned}
\label{eq:branch-curvature}
\end{equation}
Their maximizers can therefore be approximated efficiently by gradient ascent, and we use $\widehat{y}_\pm$ to denote the resulting inexact approximations. With the stepsize $\eta_\pm:=2/(m_\pm+M_\pm)$, the gradient-ascent iteration
\begin{equation}
\widehat{y}_\pm
\gets\widehat{y}_\pm+\eta_\pm
\left[
\pm\frac{g(\widehat{y}_\pm)}{\Lambda_\pm}
-(\widehat{y}_\pm-x)
\right].
\label{eq:ppd-inner-update}
\end{equation}
converges linearly to $y_\pm(x)$, with the standard estimate \cite[Theorem~2.1.15]{Nesterov2018}
\begin{equation}
\norm{\widehat{y}_\pm^{(n)}-y_\pm(x)}
\le q_\pm^n\norm{\widehat{y}_\pm^{(0)}-y_\pm(x)},
\quad
q_\pm:=\frac{M_\pm-m_\pm}{M_\pm+m_\pm}<1.
\label{eq:inner-contraction}
\end{equation}
These inexact proximal points $\widehat{y}_\pm$ then give the outer update
\begin{equation}
\widehat{v}:=w_+(\widehat{y}_+-x)+w_-(\widehat{y}_--x),
\qquad
x_{\mathrm{new}}:=x-\frac{\tau}{L_v}\widehat{v}.
\label{eq:ppd-outer-update}
\end{equation}

To quantify the accuracy of the two inner solves, choose branchwise tolerances $\delta_\pm>0$ with aggregate tolerance $\delta:=w_+\delta_++w_-\delta_-<1$, and we require, at every outer update,
\begin{equation}
\norm{\widehat{y}_\pm-y_\pm(x)}
\le\delta_\pm\norm{v(x)}.
\label{eq:relative-proximal}
\end{equation}
Their weighted combination then satisfies
\begin{equation}
\norm{\widehat{v}-v(x)}\le w_+\norm{\widehat{y}_+-y_+(x)}+w_-\norm{\widehat{y}_--y_-(x)}\le\delta\norm{v(x)}.
\label{eq:relative-field}
\end{equation}
Fix $0<\tau<\frac{1}{1+\delta}$. For each branch, $N_\pm^{\mathrm{cold}}$ is the budget for the initial solve from $\widehat{y}_\pm=x_0$, whereas $N_\pm^{\mathrm{warm}}$ is the budget for subsequent solves initialized from the approximation retained at the preceding anchor:
\begin{subequations}\label{eq:tracking-budgets}
\begin{align}
N_\pm^{\mathrm{cold}}
&:=\min\left\{n\ge0:
\left(1+\frac{(\Lambda_++\Lambda_-)w_\mp}{\widetilde{\mu}_{\min}}\right)q_\pm^n
\le\delta_\pm\right\},\\
N_\pm^{\mathrm{warm}}
&:=\min\left\{n\ge0:
\left(\delta_\pm+\frac{(1+\delta)\tau}{m_\pm L_v}\right)q_\pm^n
\le[1-(1+\delta)\tau]\delta_\pm\right\}.
\end{align}
\end{subequations}
Combining the inner iteration \eqref{eq:ppd-inner-update} with the outer update \eqref{eq:ppd-outer-update} gives the following algorithm, which we call \emph{paired proximal descent} (\PPD) and whose output we denote by $\PPD[E](x_0,\varepsilon)$.

\begin{algorithm}[H]
\caption{Paired proximal descent}
\label{alg:proximal}
\begin{algorithmic}[1]
\Statex \textbf{Input:} Initial point $x_0$; tolerance $0<\varepsilon<\norm{g(x_0)}$; admissible parameters $\Lambda_\pm,c,\delta_\pm,\tau$
\Statex \textbf{Output:} An $\varepsilon$-stationary point $x$
\State $x\gets x_0$
\For{$t=0,1,\ldots$}
    \If{$\norm{g(x)}\le\varepsilon$}
        \State \Return $x$
    \EndIf
    \If{$t=0$}
        \State $\widehat{y}_+\gets x_0$, $\widehat{y}_-\gets x_0$, and $N_\pm\gets N_\pm^{\mathrm{cold}}$
    \Else
        \State $N_\pm\gets N_\pm^{\mathrm{warm}}$
    \EndIf
    \For{$\sigma\in\{+,-\}$}
        \For{$j=1,\ldots,N_\sigma$}
            \State $\widehat{y}_\sigma\gets\widehat{y}_\sigma+\eta_\sigma\nabla_y f_{\sigma,x}(\widehat{y}_\sigma)$
        \EndFor
    \EndFor
    \State $\widehat{v}\gets w_+(\widehat{y}_+-x)+w_-(\widehat{y}_--x)$
    \State $x\gets x-(\tau/L_v)\widehat{v}$
\EndFor
\end{algorithmic}
\end{algorithm}

We first show that the prescribed budgets \eqref{eq:tracking-budgets} maintain the required proximal accuracy \eqref{eq:relative-proximal} as the anchor moves.

\begin{proposition}[Persistent proximal tracking]
\label[proposition]{prop:proximal-tracking}
The budgets in \eqref{eq:tracking-budgets} ensure the branchwise accuracy \eqref{eq:relative-proximal} at every outer update of \Cref{alg:proximal}, and consequently $\widehat{v}$ approximates $v(x)$ with relative error at most $\delta$.
\end{proposition}

\begin{proof}
We prove \eqref{eq:relative-proximal} by induction for both branches.

Since $v_+-v_-=y_+-y_-$, the identity $v=w_+v_++w_-v_-$ gives
\begin{equation*}
v_+=v+w_-(y_+-y_-),
\qquad
v_-=v-w_+(y_+-y_-).
\end{equation*}
Together with the endpoint separation \eqref{eq:proximal-endpoint-separation}, these identities give
\begin{equation*}
\norm{v_+}\le
\left(1+\frac{(\Lambda_++\Lambda_-)w_-}{\widetilde{\mu}_{\min}}\right)\norm{v},
\qquad
\norm{v_-}\le
\left(1+\frac{(\Lambda_++\Lambda_-)w_+}{\widetilde{\mu}_{\min}}\right)\norm{v}.
\end{equation*}

At the cold start, $\widehat{y}_\pm=x_0$, so the initial branch errors are $\norm{v_\pm(x_0)}$. The contraction \eqref{eq:inner-contraction} and the cold budgets in \eqref{eq:tracking-budgets} reduce them below the respective thresholds $\delta_\pm\norm{v(x_0)}$, establishing the induction base.

For the induction step, suppose that \eqref{eq:relative-proximal} holds at an anchor $x$. The outer update moves the anchor to $x_{\mathrm{new}}$ while retaining the current approximations $\widehat{y}_\pm$ as the starting points for the next inner solves. The field-error bound \eqref{eq:relative-field} first gives
\begin{equation*}
\norm{x_{\mathrm{new}}-x}\le\frac{(1+\delta)\tau}{L_v}\norm{v(x)}.
\end{equation*}
Moreover, \Cref{lem:proximal-path} gives
\begin{equation}
\norm{y_\pm(x_{\mathrm{new}})-y_\pm(x)}
\le\frac1{m_\pm}\norm{x_{\mathrm{new}}-x}.
\label{eq:proximal-target-lipschitz}
\end{equation}
Keeping the same proximal states $\widehat{y}_\pm$ after the outer update therefore gives
\begin{equation*}
\norm{\widehat{y}_\pm-y_\pm(x_{\mathrm{new}})}
\le\left[
\delta_\pm+\frac{(1+\delta)\tau}{m_\pm L_v}
\right]\norm{v(x)}.
\end{equation*}
Meanwhile, the Lipschitz continuity of $v$ gives
\begin{equation*}
\norm{v(x_{\mathrm{new}})}
\ge\norm{v(x)}-L_v\norm{x_{\mathrm{new}}-x}
\ge[1-(1+\delta)\tau]\norm{v(x)}.
\end{equation*}
Hence
\begin{equation*}
\norm{\widehat{y}_\pm-y_\pm(x_{\mathrm{new}})}
\le
\norm{v(x_{\mathrm{new}})}\cdot\left(\delta_\pm+\frac{(1+\delta)\tau}{m_\pm L_v}\right)
\Big/[1-(1+\delta)\tau].
\end{equation*}
The contraction \eqref{eq:inner-contraction} and the warm budgets in \eqref{eq:tracking-budgets} now restore \eqref{eq:relative-proximal} at $x_{\mathrm{new}}$, completing the induction.
\end{proof}

Having established the required inner accuracy, we turn to the resulting outer descent. The following lemma shows that the relative-error bound \eqref{eq:relative-field} preserves sufficient descent.

\begin{lemma}[Inexact merit descent]
\label[lemma]{lem:inexact-merit-descent}
\begin{equation}
V(x_{\mathrm{new}})
\le V(x)-\frac{\gamma_{\delta,\tau}}{L_v}\norm{\nabla V(x)}^2
=V(x)-\frac{\gamma_{\delta,\tau}}{L_v}\norm{v(x)}^2,
\label{eq:merit-descent}
\end{equation}
where
\begin{equation*}
\gamma_{\delta,\tau}:=\tau(1-\delta)-\frac{\tau^2}{2}(1-\delta)^2>0.
\end{equation*}
\end{lemma}

\begin{proof}
Since $V$ is $L_v$-smooth, its quadratic upper bound gives
\begin{align*}
V(x_{\mathrm{new}})
&\le V(x)+\langle\nabla V(x),x_{\mathrm{new}}-x\rangle
   +\frac{L_v}{2}\norm{x_{\mathrm{new}}-x}^2\\
&=V(x)-\frac{1}{L_v}
\left[
\tau\langle v,\widehat{v}\rangle-\frac{\tau^2}{2}\norm{\widehat{v}}^2
\right].
\end{align*}
Let $e:=\widehat{v}-v$, which satisfies $\norm{e}\le\delta\norm{v}$ by \eqref{eq:relative-field}. Substituting $\widehat{v}=v+e$ gives
\begin{align*}
V(x_{\mathrm{new}})
&\le V(x)-\frac{1}{L_v}
\left[\tau(1-\frac\tau2)\norm{v}^2+\tau(1-\tau)\langle v,e\rangle-\frac{\tau^2}{2}\norm{e}^2\right]\\
&\le V(x)-\frac{\gamma_{\delta,\tau}}{L_v}\norm{v}^2,
\end{align*}
which proves \eqref{eq:merit-descent}.
\end{proof}

As a corollary, telescoping \eqref{eq:merit-descent} over $T$ updates from $x_t$ gives
\begin{equation}
\begin{aligned}
\min_{t\le i<t+T}\norm{g(x_i)}^2
&\le\frac1{\underline{C}_1^2T}\sum_{i=t}^{t+T-1}\norm{\nabla V(x_i)}^2
\le\frac{L_v}{\gamma_{\delta,\tau}\underline{C}_1^2T}V(x_t)\\
&\le\frac{\overline{C}_0L_v}{\gamma_{\delta,\tau}\underline{C}_1^2T}\norm{g(x_t)}^2.
\end{aligned}
\label{eq:epoch-bound}
\end{equation}
This motivates the residual half-life
\begin{equation}
T_{1/2}:=\left\lceil
\frac{4\overline{C}_0L_v}{\gamma_{\delta,\tau}\underline{C}_1^2}
\right\rceil.
\label{eq:residual-half-life}
\end{equation}
Thus, starting from any $x_t$, the residual $\norm{g}$ is halved within at most $T_{1/2}$ outer iterations.

The branchwise budgets therefore realize the inexact merit descent using only gradient evaluations. Together with the residual half-life, they give the final whole-space complexity bound.

\begin{theorem}[Whole-space first-order tractability]
\label[theorem]{thm:main}
Let $E\in\classF_{\mu_\pm,L_\pm}(\R^d)$. Choose $\Lambda_\pm>L_\pm$, $c\in(-\mu_-,\mu_+)$, and $\delta_\pm>0$ such that $\delta:=w_+\delta_++w_-\delta_-<1$, and fix $0<\tau<1/(1+\delta)$.
For every $x_0\in\R^d$ and $0<\varepsilon\le\frac12\norm{g(x_0)}$, \Cref{alg:proximal} returns an $\varepsilon$-stationary point. If $T$ is the number of outer updates and $Q$ the number of gradient evaluations, then
\begin{equation}
\begin{aligned}
T&\le T_{1/2}\left\lceil\log_2\frac{\norm{g(x_0)}}\varepsilon\right\rceil,\\
Q&\le N_+^{\mathrm{cold}}+N_-^{\mathrm{cold}}
+\left(N_+^{\mathrm{warm}}+N_-^{\mathrm{warm}}+1\right)T.
\end{aligned}
\label{eq:exact-query-bound}
\end{equation}
Write $\mu_{\max}:=\max\{\mu_+,\mu_-\}$. For the concrete choice
\begin{equation}
\Lambda_+=\Lambda_-=2L_{\max},
\qquad
c=c_\mu,
\qquad
\delta_+=\delta_-=0.1,
\qquad
\tau=0.5,
\label{eq:concrete-ppd-parameters}
\end{equation}
the complexity parameters satisfy
\begin{equation}
T_{1/2}=O\!\left(\frac{L_{\max}^2}{\mu_+\mu_-}\right),
\qquad
N_\pm^{\mathrm{cold}}=O\!\left(\log\frac{2L_{\max}}{\mu_{\max}}\right),
\qquad
N_\pm^{\mathrm{warm}}=O(1),
\label{eq:concrete-complexity-parameters}
\end{equation}
and hence
\begin{equation}
Q=O\!\left(\frac{L_{\max}^2}{\mu_+\mu_-}\log\frac{\norm{g(x_0)}}\varepsilon\right).
\label{eq:main-query-bound}
\end{equation}
This query bound is independent of $d$ and requires no regularity beyond the defining secant inequalities.
\end{theorem}

\begin{proof}
The tracking estimate of \Cref{prop:proximal-tracking} ensures that the residual half-life applies from every monitored iterate. Applying it successively, after $\lceil\log_2(\norm{g(x_0)}/\varepsilon)\rceil$ halvings the target accuracy is reached, proving the stated outer-iteration bound. The two cold solves cost $N_+^{\mathrm{cold}}+N_-^{\mathrm{cold}}$ gradient evaluations, each warm refinement costs $N_+^{\mathrm{warm}}+N_-^{\mathrm{warm}}$, and monitoring costs at most one additional evaluation per outer update. This proves \eqref{eq:exact-query-bound}.

We now specialize to \eqref{eq:concrete-ppd-parameters}. Then $\widetilde{\mu}_+=\widetilde{\mu}_-=\bar{\mu}$, while
\begin{equation*}
\frac12\le m_\pm<1,
\qquad
1<M_\pm\le\frac32,
\qquad
q_\pm\le\frac12,
\qquad
L_v\le3.
\end{equation*}
The constants in \Cref{thm:paired-proximal-equivalence} therefore satisfy
\begin{equation*}
\overline{C}_0=O\!\left(\frac{\bar{\mu}}{\mu_{\min}L_{\max}^2}\right),
\qquad
\underline{C}_1=\Omega\!\left(\frac{\bar{\mu}}{L_{\max}^2}\right).
\end{equation*}
Since $\bar{\mu}\mu_{\min}\ge\mu_+\mu_-/2$, it follows that
\begin{equation*}
\frac{\overline{C}_0L_v}{\underline{C}_1^2}
=O\!\left(\frac{L_{\max}^2}{\mu_+\mu_-}\right),
\end{equation*}
and the definition \eqref{eq:residual-half-life} gives the stated bound on $T_{1/2}$.

In addition, the factors multiplying $q_\pm^n$ in the cold budgets are $1+4L_{\max}w_\mp/\bar{\mu}=O(L_{\max}/\bar{\mu})=O(L_{\max}/\mu_{\max})$. Since $q_\pm\le1/2$ and all remaining quantities in the warm budgets are bounded by absolute constants, the budgets given in \eqref{eq:tracking-budgets} imply the two inner-budget estimates in \eqref{eq:concrete-complexity-parameters}. Combining these estimates with \eqref{eq:exact-query-bound} yields the overall complexity bound \eqref{eq:main-query-bound} and completes the proof.
\end{proof}

The theorem therefore establishes dimension-free first-order tractability on $\R^d$ from arbitrary initialization. We conclude this section with two remarks connecting its complexity bound to familiar first-order benchmarks.

\begin{remark}[Strongly convex specialization]
If $E$ is $\mu$-strongly convex and $L$-smooth, the negative spectral interval is inactive. Taking $\mu_+=\mu$, $L_+=L$, and $\mu_-=L_-=L$ gives $E\in\classF_{\mu_\pm,L_\pm}(\R^d)$ and
\begin{equation*}
\frac{L_{\max}^2}{\mu_+\mu_-}=\frac{L}{\mu},
\qquad
Q=O\!\left(\frac{L}{\mu}\log\frac{\norm{g(x_0)}}\varepsilon\right).
\end{equation*}
Thus the general bound scales linearly with the usual strongly convex condition number $L/\mu$, as does standard gradient descent.
\end{remark}

\begin{remark}[Deterministic lower benchmark]
\label[remark]{rem:whole-space-lower-bound}
Set $\kappa_\pm:=L_\pm/\mu_\pm$ and suppose that $\kappa_+\kappa_->1$. If $0<\varepsilon\le\norm{g(x_0)}/2$, then, provided the dimension is sufficiently large, every deterministic first-order method that guarantees an $\varepsilon$-stationary point throughout $\classF_{\mu_\pm,L_\pm}(\R^d)$ has worst-case query count
\begin{equation}
Q\ge
\frac{1}{\psi(\sqrt{\kappa_+\kappa_-})}
\log\frac{\norm{g(x_0)}}\varepsilon,
\label{eq:whole-space-lower-benchmark}
\end{equation}
where $\psi(z):=\log((z+1)/(z-1))$.
Thus the logarithmic dependence on relative accuracy in \Cref{thm:main} is unavoidable for deterministic first-order methods. The proof, precise dimension requirement, and a sharper endpoint-sensitive bound are given in \Cref{app:whole-space-lower-bound}; determining the optimal dependence on the spectral endpoints remains open.
\end{remark}

\section{First-order tractability under domain restrictions}
\label{sec:query-restrictions}

We now replace $\R^d$ by an open domain $\X\subset\R^d$ and ask whether the dimension-free tractability of \Cref{thm:main} persists.
Define the corresponding intrinsic function class
\begin{equation*}
\classF_{\mu_\pm,L_\pm}(\X):=
\bigl\{E\in C^1(\X):\ \eqref{eq:exact-secant}\text{ holds for every }x_1,x_2\in\X\bigr\}.
\end{equation*}
Thus the geometric assumption is unchanged; only its domain of validity is restricted. The computational effect is nevertheless substantial because oracle information can now be acquired only inside $\X$.

We adopt the standard domain-access convention of constrained optimization: the domain $\X$ and its geometry are part of the problem data, so membership, containment, and distance information involving $\X$ is available without charge. The first-order oracle returns $(E(x),g(x))$ only for $x\in\X$, and query complexity counts these oracle evaluations. Accordingly, both information acquisition and output are confined to $\X$, giving rise to the following notions of feasibility.

\begin{definition}[Feasibility and promised feasibility]
\label[definition]{def:feasibility-notions}
A point is \emph{feasible} if it belongs to $\X$, and an oracle algorithm is feasible if every query point and its final output are feasible. An instance $E\in\classF_{\mu_\pm,L_\pm}(\X)$ has \emph{promised feasibility} if
\begin{equation*}
\exists\,x^\star\in\X
\quad\text{such that}\quad
g(x^\star)=0.
\end{equation*}
\end{definition}

The lower estimate in \eqref{eq:ordinary-bilip} makes the stationary point unique whenever it exists. For a promised-feasible instance, our goal remains to construct a feasible algorithm returning an $\varepsilon$-stationary point $\widehat{x}$, meaning that $\norm{g(\widehat{x})}\le\varepsilon$.

\subsection{Information barrier under promised feasibility}
\label{sec:promised-feasibility}

Promised feasibility guarantees a unique stationary point, while uniform nondegeneracy retains quantitative metric control. Nevertheless, dimension-free tractability can fail dramatically: on the unit ball $\X:=B(0,1)$, finding an approximate stationary point may require exponentially many queries in the dimension. This curse of dimensionality persists even under a much stronger infinite-order oracle, which, for a smooth instance $E\in C^\infty(\X)$, returns the complete jet
\begin{equation*}
J^\infty E(x):=\bigl(E(x),\nabla E(x),\nabla^2E(x),\ldots\bigr)
\end{equation*}
at each query $x\in\X$. Thus the lower bound is informational rather than first-order: it applies equally to Newton and arbitrary higher-order local methods.

Following the index convention of \Cref{sec:class}, we write $\classF_{\mu,L}^{[k]}(\X)$ for the index-$k$ subclass. For a smooth member, this means that its Hessian has $k$ negative eigenvalues throughout $\X$.

\begin{theorem}[Curse of dimensionality under promised feasibility]
\label[theorem]{thm:promised-complexity}
There exist universal constants $\Cacc,\Cdim>0$ such that, for every $d\ge3$, every $0<k<d$, and every $0<\mu<L$ with $L/\mu\ge6$, any deterministic algorithm guaranteed to find a $(\Cacc\,\mu)$-stationary point for every promised-feasible smooth member $E\in C^\infty(\X)\cap\classF_{\mu,L}^{[k]}(\X)$ has worst-case query count $Q$ satisfying
\begin{equation}
Q\ge e^{\Cdim\,d}.
\label{eq:deterministic-lower-bound}
\end{equation}
\end{theorem}

The whole-space and restricted-domain theories therefore diverge sharply. On $\R^d$, \Cref{sec:proximal} shows that uniform nondegeneracy alone supports dimension-free computation without a prescribed curvature orientation. Once oracle access is restricted, however, the spectral endpoint bounds and a promise of existence supply no usable localization from the observed transcript; \Cref{thm:promised-complexity} shows that they no longer ensure dimension-free tractability. This contrasts with smooth strongly convex minimization and SCSC minimax optimization over a convex feasible set, where projected gradient \cite[Sec.~10.6]{Beck2017} and projected extragradient \cite{Tseng1995} exploit the corresponding strongly monotone formulations to retain dimension-free linear convergence. The remainder of this subsection exposes the mechanism behind this information barrier through a constructive proof of \Cref{thm:promised-complexity}.

For clarity, we begin with the special case $k=1$. The core idea is to construct a promised-feasible family such that fewer than $e^{\Cdim\,d}$ queries leave two instances indistinguishable, forcing any deterministic algorithm to fail on at least one.

To prepare the construction, fix $\theta:=10^{-3}$ and set $\Cacc:=(1+\sqrt2)\theta$ and $\Cdim:=\theta^2/6$.
Choose any smooth ReLU-type ramp $\phi\in C^\infty(\R)$ satisfying
\begin{equation*}
\phi(\xi)=0\quad(\xi\le0),
\qquad
\phi'(\xi)=1\quad(\xi\ge1),
\qquad
\norm{\phi''}_\infty\le1+\theta.
\end{equation*}

For the index-one construction, write its argument in the splitting form
\begin{equation*}
z=(z_\parallel,z_\perp)\in\R\oplus\R^{d-1}.
\end{equation*}

The hard-instance family is parametrized by a hidden direction $u\in\mathbb{S}(\R^{d-1})$ through
\begin{equation}
E_u(z):=E_{\mathrm{ref}}(z)-\lambda_0\theta^2(1-\theta^2)\phi\bigl(\xi_u(z)\bigr),
\label{eq:lb-family}
\end{equation}
where $\lambda_0:=\sqrt{\mu L}$ and $E_{\mathrm{ref}}$ is the reference quadratic
\begin{equation*}
E_{\mathrm{ref}}(z):=\frac{\lambda_0}{2}\left[-\bigl(z_\parallel-(1+\theta)\bigr)^2+\norm{z_\perp}^2\right],
\end{equation*}
while the affine coordinate $\xi_u$ carries the dependence on $u$:
\begin{equation*}
\xi_u(z):=\frac{z_\parallel+\langle u,z_\perp\rangle-(1+\theta^2)}{\theta(1-\theta)}.
\end{equation*}
Thus the perturbation in \eqref{eq:lb-family} has the familiar form of a single neuron with a smooth ReLU-type activation $\phi$ applied to the affine coordinate $\xi_u$. The key feature is that this neuron is inactive outside the cap
\begin{equation*}
\mathcal{A}_u:=\{z\in\X:\xi_u(z)>0\}.
\end{equation*}
Indeed, because $\phi$ vanishes on $(-\infty,0]$, every $z\in\X\setminus\mathcal{A}_u$ satisfies
\begin{equation}
J^\infty E_u(z)=J^\infty E_{\mathrm{ref}}(z).
\label{eq:lb-oracle-hiding}
\end{equation}
Consequently, a query outside $\mathcal{A}_u$ yields exactly the same oracle response as the reference quadratic $E_{\mathrm{ref}}$ and therefore reveals no information about the hidden direction $u$ at all. To convert this oracle hiding into a complexity lower bound, we need only three further properties of the family.

\begin{proposition}[Index-one hidden-direction family]
\label[proposition]{prop:hidden-direction}
\begin{enumerate}
\item Every member of the family is smooth, belongs to the index-one class, and has the unique stationary point
\begin{equation*}
E_u\in\classF_{\mu,L}^{[1]}(\R^d),
\qquad
z_*(u)=\bigl(1-\theta^2,\theta(1+\theta)u\bigr)\in\mathcal{A}_u.
\end{equation*}

\item \emph{Accuracy requires activation.} Every $z\in\X$ satisfying $\norm{\nabla E_u(z)}\le\Cacc\,\mu$ must belong to $\mathcal{A}_u$.

\item \emph{Paired caps.} For every $u\in\mathbb{S}(\R^{d-1})$, the caps $\mathcal{A}_u$ and $\mathcal{A}_{-u}$ are disjoint. Moreover, for any $Q<e^{2\Cdim\,d}$ and any query points $z_1,\ldots,z_Q\in\X$, there exists a direction $u\in\mathbb{S}(\R^{d-1})$ such that
\begin{equation*}
z_i\notin\mathcal{A}_u\cup\mathcal{A}_{-u},
\qquad \forall i=1,\ldots,Q.
\end{equation*}
\end{enumerate}
\end{proposition}

\begin{proof}
All three assertions follow from direct calculations, which we give in turn.

We begin with admissibility and promised feasibility. Let $n_u:=\frac{1}{\sqrt{2}}(1,u)$ be the unit normal to the level sets of $\xi_u$. Differentiation gives
\begin{equation*}
\begin{aligned}
\nabla E_u(z)
&=\lambda_0(1+\theta-z_\parallel,z_\perp)
-\sqrt2\lambda_0\theta(1+\theta)\phi'(\xi_u(z))n_u,\\
\nabla^2E_u(z)
&=\lambda_0\diag(-1,I_{d-1})
-2\lambda_0\frac{1+\theta}{1-\theta}\phi''(\xi_u(z))n_un_u^\top.
\end{aligned}
\end{equation*}

The perturbation acts only on the coordinate plane spanned by $(1,0)$ and $(0,u)$, so the spectral enclosure reduces to a $2\times2$ calculation. On this plane, the Hessian $\nabla^2E_u(z)$ has the block
\begin{equation*}
\begin{pmatrix}
-\lambda_0+b_u(z)/2&b_u(z)/2\\
b_u(z)/2&\lambda_0+b_u(z)/2
\end{pmatrix},
\qquad
b_u(z):=-2\lambda_0\frac{1+\theta}{1-\theta}\phi''(\xi_u(z)),
\end{equation*}
while every remaining eigenvalue equals $\lambda_0\in[\mu,L]$.
The choice of $\phi$ gives
\begin{equation*}
|b_u(z)|
\le2\sqrt{\mu L}\,\frac{1+\theta}{1-\theta}\norm{\phi''}_\infty
\le2\sqrt{\mu L}\,\frac{(1+\theta)^2}{1-\theta}
<L-\mu,
\end{equation*}
where the last inequality follows from $\theta=10^{-3}$ and $L/\mu\ge6$. 
The negative determinant shows that the block has one negative and one positive eigenvalue; write them as $-\lambda_-$ and $\lambda_+$ with $\lambda_\pm>0$. Its trace and determinant give
\begin{equation*}
\lambda_+-\lambda_-=b_u(z),
\qquad
\lambda_+\lambda_-=\mu L.
\end{equation*}
The product identity shows that if either $\lambda_\pm$ lies below $\mu$ or above $L$, then the other lies beyond the opposite endpoint. This would give $|\lambda_+-\lambda_-|>L-\mu$, contradicting $|b_u(z)|<L-\mu$. Hence $\lambda_\pm\in[\mu,L]$, so the block has one eigenvalue in $[-L,-\mu]$ and one in $[\mu,L]$, proving the asserted spectral enclosure and index.

We now verify the claimed stationary point. For $u\in\mathbb{S}(\R^{d-1})$, direct substitution of $z_*(u)$ gives $\xi_u(z_*(u))=1$ and $\nabla E_u(z_*(u))=0$, while
\begin{equation*}
\norm{z_*(u)}=\sqrt{1-\theta^2+2\theta^3+2\theta^4}\le1-\theta^2/4<1.
\end{equation*}
Hence $z_*(u)\in\mathcal{A}_u\subset\X$ is the unique stationary point.

We next verify that accuracy requires activation: whenever $z\in\X\setminus\mathcal{A}_u$, by \eqref{eq:lb-oracle-hiding},
\begin{equation*}
\norm{\nabla E_u(z)}
=\norm{\nabla E_{\mathrm{ref}}(z)}
=\lambda_0\norm{z-(1+\theta,0)}
>\lambda_0\theta
\ge\sqrt6\,\theta\mu
>\Cacc\,\mu.
\end{equation*}

Finally, we prove the paired-cap property. The two caps are disjoint: for any $z\in\mathcal{A}_u\cap\mathcal{A}_{-u}$, we have
\begin{equation*}
z_\parallel+\langle u,z_\perp\rangle>1+\theta^2,
\qquad
z_\parallel-\langle u,z_\perp\rangle>1+\theta^2,
\end{equation*}
and hence $z_\parallel>1+\theta^2$, which contradicts $z\in\X$.

Now fix a query point $z\in\X$. If $z\in\mathcal{A}_u$, then $z_\perp\ne0$ and
\begin{equation*}
\left\langle u,\frac{z_\perp}{\norm{z_\perp}}\right\rangle=\frac{\langle u,z_\perp\rangle}{\norm{z_\perp}}
>\frac{(1-z_\parallel)+\theta^2}
{\sqrt{(1-z_\parallel)(1+z_\parallel)}}
\ge\frac{2\theta}{\sqrt{1+z_\parallel}}
\ge\sqrt2\,\theta.
\end{equation*}
The first inequality uses $z\in\mathcal{A}_u$ and $\norm{z}<1$, while the remaining inequalities follow from the arithmetic--geometric mean inequality and $z_\parallel<1$.

Writing $\sigma_{d-2}$ for normalized surface measure on $\mathbb{S}(\R^{d-1})$, let $Z=(Z_1,\ldots,Z_{d-1})$ be standard Gaussian, so that $Z/\norm{Z}$ is uniform on this sphere. With $S:=\sum_{j=2}^{d-1}Z_j^2\sim\chi^2_{d-2}$, the elementary Gaussian tail bound $\mathbb{P}(Z_1>t)\le\frac12e^{-t^2/2}$ gives
\begin{equation*}
\begin{aligned}
\sigma_{d-2}\bigl(\{u:z\in\mathcal{A}_u\}\bigr)
&\le\mathbb{P}\!\left(Z_1>\frac{\sqrt2\,\theta}{\sqrt{1-2\theta^2}}\sqrt{S}\right)
\le\frac12\mathbb{E}\exp\!\left(-\frac{\theta^2S}{1-2\theta^2}\right)\\
&=\frac12(1-2\theta^2)^{(d-2)/2}\le\frac12e^{-2\Cdim\,d}.
\end{aligned}
\end{equation*}
The directions for which $z\in\mathcal{A}_u\cup\mathcal{A}_{-u}$ consequently have measure at most $e^{-2\Cdim\,d}$. A union bound over $Q<e^{2\Cdim\,d}$ query points leaves some $u$ for which every query misses both caps, proving the final assertion.
\end{proof}

Together, the last two properties expose the information barrier. Until a query enters one of their disjoint active caps, the two promised-feasible instances $E_u$ and $E_{-u}$ are indistinguishable; yet no point can be sufficiently stationary for both.

\begingroup
\renewcommand{\proofname}{Proof of \Cref{thm:promised-complexity}}
\begin{proof}
We first prove the claim for $k=1$ through an adversarial reference transcript. Suppose that a deterministic algorithm has query budget $Q<e^{\Cdim\,d}$. We answer its queries provisionally with $J^\infty E_{\mathrm{ref}}$ and consider the resulting dichotomy.

If the algorithm does not return a feasible candidate within $Q$ queries---whether it continues querying or halts without one---the paired-cap property gives a direction $u$ for which all observed responses coincide with those of the promised-feasible instance $E_u$. The algorithm behaves identically on $E_u$ and therefore violates its guarantee.

Otherwise, it returns some $z_{\mathrm{out}}\in\X$ after at most $Q$ queries. The paired-cap property gives a direction $u$ for which every query misses both $\mathcal{A}_u$ and $\mathcal{A}_{-u}$, so \eqref{eq:lb-oracle-hiding} makes the observed transcript valid for both $E_u$ and $E_{-u}$. Determinism forces the same output on both instances. If $z_{\mathrm{out}}$ were $(\Cacc\,\mu)$-stationary for both, the second property of \Cref{prop:hidden-direction} would place it in the disjoint intersection $\mathcal{A}_u\cap\mathcal{A}_{-u}$. Thus the algorithm fails on at least one instance, proving the claim for $k=1$.

The remaining indices require only definite quadratic padding. Set 
\begin{equation*}
    m:=\max\{k,d-k\}+1\ge \lceil d/2\rceil+1\ge3.
\end{equation*}
Let $E_u^{(m)}$ denote the index-one hidden-direction family in $\R^m$, and, for $x=(z,\zeta)\in\R^m\oplus\R^{d-m}$, define
\begin{equation*}
E_{k,u}(z,\zeta):=
\begin{cases}
E_u^{(m)}(z)-\dfrac{\lambda_0}{2}\norm{\zeta}^2,&k\le d/2,\\
-E_u^{(m)}(z)+\dfrac{\lambda_0}{2}\norm{\zeta}^2,&k>d/2,
\end{cases}.
\end{equation*}
If $k\le d/2$, the negative quadratic contributes $d-m=k-1$ additional negative directions. If $k>d/2$, then $m=k+1$, so $-E_u^{(m)}$ already has index $k$. Hence in both cases $E_{k,u}\in\classF_{\mu,L}^{[k]}(\R^d)$, with stationary point $x^\star=(z_*^{(m)}(u),0)\in\X$.

Replacing $E_u^{(m)}$ by its reference quadratic gives a common reference transcript. Outside the lifted cap
\begin{equation*}
\mathcal{A}_{k,u}:=\{(z,\zeta)\in\X:z\in\mathcal{A}_u^{(m)}\}
\end{equation*}
the infinite jets agree and $\norm{\nabla E_{k,u}}>\Cacc\,\mu$, while the lifted caps for $u$ and $-u$ remain disjoint. Since $m>d/2$, the paired-cap estimate in dimension $m$ applies to every $Q<e^{\Cdim\,d}$. Repeating the index-one transcript argument proves \eqref{eq:deterministic-lower-bound} for every $0<k<d$.
\end{proof}
\endgroup

\begin{remark}[Sharpness in dimension]
For fixed spectral endpoints and fixed accuracy, the exponential dependence on $d$ is sharp. Indeed, querying a $(\Cacc\,\mu/L)$-net of the unit ball gives an $\exp(O(d))$ upper bound by \eqref{eq:ordinary-bilip}. Hence the deterministic fixed-accuracy worst-case query complexity is $\exp(\Theta(d))$. The numerical constants in the construction are not optimized; sharpening them would not alter this dimensional scaling.
\end{remark}

The construction exposes how little information the promise provides. Until a query enters an active cap, even the complete infinite-order transcript is independent of $u$: it reveals nothing at all about where, among exponentially many indistinguishable possibilities, the stationary point actually lies. Thus promised feasibility only guarantees existence, but supplies no usable localization beyond $\X$ itself. Recovering dimension-free tractability requires an observable region that actually confines the target.

\subsection{Dimension-free tractability under certified feasibility}
\label{sec:certified-feasibility}

We now return to a general open domain $\X\subset\R^d$. The first-order geometry of \Cref{sec:class} provides exactly such an observable localization. At any feasible point $x\in\X$, the certification ball $\Bcert(x)$ can be constructed from $x$, the observed gradient $g(x)$, and the inner endpoints $\mu_\pm$; whenever a stationary point exists, \eqref{eq:root-certification-ball} places it inside this ball. This motivates requiring the entire observable localization to remain within $\X$.

\begin{definition}[Certified feasibility]
\label[definition]{def:certified-feasibility}
A feasible point $x\in\X$ is \emph{certified feasible} if
\begin{equation*}
\Bcert(x)\subset\X.
\end{equation*}
\end{definition}

\begin{proposition}[Certification implies promised feasibility]
\label[proposition]{prop:certification-implies-promise}
Let $E\in\classF_{\mu_\pm,L_\pm}(\X)$. If $x\in\X$ is certified feasible, then there exists a unique $x^\star\in\X$ such that $g(x^\star)=0$, and moreover $x^\star\in\Bcert(x)$.
\end{proposition}

\begin{proof}
The lower secant estimate localizes any zero and makes it unique, but it does not guarantee existence on the restricted domain. Since the whole-space global-inversion argument is unavailable, we establish existence by a classical Brouwer degree argument \cite[Chap.~1]{Deimling1985}. For $0\le\beta\le1$, define the homotopy
\begin{equation*}
\tilde{g}_\beta:=g-(1-\beta)g(x).
\end{equation*}
At $\beta=0$, the map $\tilde{g}_0=g-g(x)$ has the unique zero $x$, while $\tilde{g}_1=g$ is the target stationary equation.

We first show that, for every $\beta\in[0,1]$, any zero of $\tilde{g}_\beta$ must lie in $\Bcert(x)$. Indeed, if $\tilde{g}_\beta(z)=0$, then the lower secant estimate \eqref{eq:exact-secant-lower}, applied to $x$ and $z$, gives
\begin{equation}
z\in(1-\beta)\{x\}+\beta\Bcert(x)\subseteq\Bcert(x).
\label{eq:certified-homotopy-localization}
\end{equation}

Now choose a bounded open neighborhood $U$ of $\Bcert(x)$ such that $\Bcert(x)\subset U$ and $\overline{U}\subset\X$. Since $\tilde{g}_\beta$ has no zero on $\partial U$, the Brouwer degree $\deg(\tilde{g}_\beta,U,0)$ is well defined for every $\beta\in[0,1]$. At $\beta=0$, invariance of domain gives $\deg(\tilde{g}_0,U,0)=\pm1$. Homotopy invariance therefore gives $\deg(\tilde{g}_\beta,U,0)=\pm1$ for every $\beta$, so each $\tilde{g}_\beta$ has a zero in $U$. The lower estimate in \eqref{eq:ordinary-bilip} makes this zero unique in $\X$, while \eqref{eq:certified-homotopy-localization} places it in $\Bcert(x)$. Taking $\beta=1$ proves the claim.
\end{proof}

Let $x_0\in\X$ be certified feasible. By \Cref{prop:certification-implies-promise}, the ball $\Bcert(x_0)$ contains the unique stationary point $x^\star$. The homotopy proof gives considerably more: for every $0\le\beta\le1$, there is a unique point $x_\beta^\dagger\in\Bcert(x_0)$ satisfying
\begin{equation*}
g(x_\beta^\dagger)=(1-\beta)g(x_0),
\qquad
x_0^\dagger=x_0,
\qquad
x_1^\dagger=x^\star.
\end{equation*}
Thus the initial certification ball $\Bcert(x_0)$ is more than a static enclosure of $x^\star$: it contains a canonical path from $x_0$ to $x^\star$ along which the gradient contracts exactly and linearly to zero. The following proposition reveals the matching spatial refinement: the certification balls along this path are nested and ultimately collapse to $\{x^\star\}$.

\begin{proposition}[Shrinking certification balls]
\label[proposition]{prop:nested-certification}
Let $E\in\classF_{\mu_\pm,L_\pm}(\X)$, let $x_0\in\X$ satisfy $\Bcert(x_0)\subset\X$, and let $x_\beta^\dagger$ be the intermediate point defined above. Then for any $0\le\beta_1\le\beta_2\le1$, the corresponding certification balls satisfy
\begin{equation}
\Bcert(x_0)
\supseteq\Bcert(x_{\beta_1}^\dagger)
\supseteq\Bcert(x_{\beta_2}^\dagger)
\supseteq\Bcert(x_1^\dagger)
=\{x^\star\}.
\label{eq:nested-certification}
\end{equation}
Consequently, every $x_\beta^\dagger$ is certified feasible.
\end{proposition}

\begin{proof}
For $0\le\beta_1<\beta_2\le1$, set $x_i:=x_{\beta_i}^\dagger$, $p_i:=g(x_i)$, $\Delta x:=x_1-x_2$, and $\Delta p:=p_1-p_2$. Recall from \eqref{eq:reciprocal-certification-ball} that $\Bcert(x_i)=\overline{B}(z_i,r_i)$ with $z_i=x_i-c_L^\sharp p_i$ and $r_i=\bar{L}^\sharp\norm{p_i}$. Since $p_i=(1-\beta_i)g(x_0)$, the two gradients lie on the same ray and $\bar{L}^\sharp\norm{\Delta p}=r_1-r_2$.

The lower secant estimate \eqref{eq:exact-secant-lower} then gives
\begin{equation*}
\norm{z_1-z_2}
=\norm{\Delta x-c_L^\sharp\Delta p}
\le\bar{L}^\sharp\norm{\Delta p}
=r_1-r_2.
\end{equation*}
Thus $\norm{z_1-z_2}+r_2\le r_1$, which proves $\Bcert(x_2)\subseteq\Bcert(x_1)$. Together with $x_0^\dagger=x_0$, this establishes the nested structure \eqref{eq:nested-certification} and shows that every $x_\beta^\dagger$ is certified feasible. Finally, $g(x_1^\dagger)=0$ makes $\Bcert(x_1^\dagger)=\{x_1^\dagger\}=\{x^\star\}$.
\end{proof}

The crucial link with \Cref{sec:proximal} is that affine tilting preserves the secant geometry. For $0\le\beta\le1$, define
\begin{equation*}
\tilde{E}_\beta:=E-(1-\beta)\langle g(x_0),\mathord\cdot\rangle,
\qquad
\nabla\tilde{E}_\beta=g-(1-\beta)g(x_0).
\end{equation*}
Each $\tilde{E}_\beta$ belongs to $\classF_{\mu_\pm,L_\pm}(\X)$, has the unique stationary point $x_\beta^\dagger$, and can be queried directly through the oracle for $E$. We therefore discretize $[0,1]$ by $\beta_j:=j/N_{\mathrm{mesh}}$, $0\le j\le N_{\mathrm{mesh}}$, and apply \PPD{} successively to the corresponding tilted objectives. By \Cref{prop:nested-certification}, the exact continuation points remain certified feasible along this mesh. This gives the following algorithm.

\begin{algorithm}
\caption{Certified continuation}
\label{alg:continuation}
\begin{algorithmic}[1]
\Statex \textbf{Input:} Initial certified-feasible point $x_0$; final tolerance $\varepsilon$; integer $N_{\mathrm{mesh}}\ge1$
\Statex \textbf{Output:} An $\varepsilon$-stationary point $\widehat{x}\in\X$
\For{$j=1,\ldots,N_{\mathrm{mesh}}$}
    \State $\beta_j\gets j/N_{\mathrm{mesh}}$
    \State $\tilde{E}_{\beta_j}\gets E-(1-\beta_j)\langle g(x_0),\mathord\cdot\rangle$
    \If{$j<N_{\mathrm{mesh}}$}
        \State $\varepsilon_j\gets\frac12\norm{\nabla\tilde{E}_{\beta_j}(x_{j-1})}$
    \Else
        \State $\varepsilon_j\gets\varepsilon$
    \EndIf
    \State $x_j\gets\PPD[\tilde{E}_{\beta_j}](x_{j-1},\varepsilon_j)$
\EndFor
\State \Return $x_{N_{\mathrm{mesh}}}$
\end{algorithmic}
\end{algorithm}

The continuation in \Cref{alg:continuation} is governed by a repeated double-then-halve mechanism for the translated residual. To make this precise, set $p_0:=g(x_0)$ and define the residual mesh size
\begin{equation*}
\varepsilon_{\mathrm{mesh}}:=\norm{p_0}/N_{\mathrm{mesh}}.
\end{equation*}

\begin{lemma}[Continuation residual recurrence]
\label[lemma]{lem:continuation-residual}
\begin{equation}
\begin{aligned}
\norm{\nabla\tilde{E}_{\beta_j}(x_{j-1})}
&\le2\varepsilon_{\mathrm{mesh}},
&&1\le j\le N_{\mathrm{mesh}},\\
\norm{\nabla\tilde{E}_{\beta_j}(x_j)}
&\le\varepsilon_{\mathrm{mesh}},
&&1\le j<N_{\mathrm{mesh}}.
\end{aligned}
\label{eq:continuation-residual-bounds}
\end{equation}
\end{lemma}

\begin{proof}
We prove both bounds by induction on $j$. At level zero, $\nabla\tilde{E}_{\beta_0}(x_0)=0$, while consecutive targets differ by
\begin{equation*}
\norm{\nabla\tilde{E}_{\beta_j}(x)-\nabla\tilde{E}_{\beta_{j-1}}(x)}
=\frac{\norm{p_0}}{N_{\mathrm{mesh}}}=\varepsilon_{\mathrm{mesh}}.
\qquad x\in\X,
\end{equation*}
Hence, if the output residual at level $\beta_{j-1}$ is at most $\varepsilon_{\mathrm{mesh}}$, then
\begin{align*}
\norm{\nabla\tilde{E}_{\beta_j}(x_{j-1})}
&\le\norm{\nabla\tilde{E}_{\beta_{j-1}}(x_{j-1})}+\varepsilon_{\mathrm{mesh}}
\le2\varepsilon_{\mathrm{mesh}},\\
\norm{\nabla\tilde{E}_{\beta_j}(x_j)}
&\le\varepsilon_j=\frac12\norm{\nabla\tilde{E}_{\beta_j}(x_{j-1})}
\le\varepsilon_{\mathrm{mesh}},
\end{align*}
where the second line applies for $j<N_{\mathrm{mesh}}$. Starting from the zero residual at $\beta_0$, these two estimates close the induction and prove \eqref{eq:continuation-residual-bounds}.
\end{proof}

Recall from \eqref{eq:residual-half-life} that \PPD{} takes at most $T_{1/2}$ outer iterations to halve the residual. If every \PPD{} run in \Cref{alg:continuation} is feasible, then under the concrete parameters of \Cref{thm:main} its total query count is bounded by
\begin{equation*}
Q=O\!\left(
\frac{L_{\max}^2}{\mu_+\mu_-}
\left[
N_{\mathrm{mesh}}+\log\frac{\norm{g(x_0)}}\varepsilon
\right]
\right).
\end{equation*}
Compared with the whole-space bound \eqref{eq:main-query-bound}, the only additional term is $N_{\mathrm{mesh}}$, which accounts for traversing the continuation levels.

It remains to choose $N_{\mathrm{mesh}}$ so that every such run is feasible. The following lemma provides a sufficient condition for feasibility on $\X$.

\begin{lemma}[Spatial confinement]
\label[lemma]{lem:spatial-confinement}
For every admissible choice of $\Lambda_\pm,c,\delta_\pm,\tau$, there exists a spatial-confinement factor $\Gamma>0$, depending only on these parameters and the spectral endpoints $\mu_\pm,L_\pm$, with the following property. For any $h\in\classF_{\mu_\pm,L_\pm}(\X)$ with stationary point $x^\dagger\in\X$ and any $x_{\mathrm{in}}\in\X$, the \PPD{} run for $h$ initialized at $x_{\mathrm{in}}$ is feasible whenever
\begin{equation}
\overline{B}\!\left(x^\dagger,\Gamma\norm{\nabla h(x_{\mathrm{in}})}\right)\subset\X.
\label{eq:certification-confinement-condition}
\end{equation}
For this run, the convergence and query bounds of \Cref{thm:main} hold unchanged.
\end{lemma}

The proof first re-establishes the proximal geometry and merit estimates of \Cref{sec:proximal} on a suitable convex neighborhood contained in $\X$, then combines merit descent with the inner-iteration contraction in a first-exit argument that confines every oracle query. The complete argument, including an explicit choice of $\Gamma$, is deferred to \Cref{app:spatial-confinement}.

We now apply the lemma along the continuation mesh. For any certified-feasible point $x$, define its boundary clearance by
\begin{equation*}
\clr_{\X}(x)
:=\dist\bigl(\Bcert(x),\R^d\setminus\X\bigr),
\end{equation*}
with the convention $\clr_{\R^d}(x):=+\infty$. By \Cref{prop:nested-certification}, every exact continuation point $x_\beta^\dagger$ lies in the compact initial certificate $\Bcert(x_0)$, and certified feasibility gives $\clr_{\X}(x_0)>0$.

At level $\beta_j$, \Cref{lem:continuation-residual} bounds the input residual by $2\varepsilon_{\mathrm{mesh}}$. Since $x_{\beta_j}^\dagger\in\Bcert(x_0)$, \Cref{lem:spatial-confinement} makes the entire $j$th run feasible whenever
\begin{equation*}
2\Gamma\varepsilon_{\mathrm{mesh}}<\clr_{\X}(x_0).
\end{equation*}
Using $\varepsilon_{\mathrm{mesh}}=\norm{p_0}/N_{\mathrm{mesh}}$, it is therefore sufficient to choose $N_{\mathrm{mesh}}$ as the smallest integer satisfying
\begin{equation}
N_{\mathrm{mesh}}>
2\Gamma\frac{\norm{g(x_0)}}{\clr_{\X}(x_0)},
\label{eq:continuation-horizon}
\end{equation}
where the ratio is interpreted as zero when $\clr_{\X}(x_0)=+\infty$.

Combining this feasible continuation length with the preceding conditional query estimate yields the complete certified-domain complexity theorem.

\begin{theorem}[Dimension-free continuation under certified feasibility]
\label[theorem]{thm:certified-continuation}
Let $E\in\classF_{\mu_\pm,L_\pm}(\X)$, and let $x_0\in\X$ be certified feasible. For any $0<\varepsilon<\norm{g(x_0)}$, choose the concrete \PPD{} parameters in \eqref{eq:concrete-ppd-parameters}, let $\Gamma$ be the corresponding factor in \Cref{lem:spatial-confinement}, and let $N_{\mathrm{mesh}}$ be the smallest integer satisfying \eqref{eq:continuation-horizon}. Then \Cref{alg:continuation} is feasible, returns $\widehat{x}$ satisfying $\norm{g(\widehat{x})}\le\varepsilon$, and has query complexity
\begin{equation*}
Q=O\!\left(
\frac{L_{\max}^2}{\mu_+\mu_-}
\left[
1+\Gamma\frac{\norm{g(x_0)}}{\clr_{\X}(x_0)}
+\log\frac{\norm{g(x_0)}}{\varepsilon}
\right]
\right).
\end{equation*}
\end{theorem}

\begin{corollary}[Whole-space reduction]
\label[corollary]{cor:whole-space-reduction}
When $\X=\R^d$, one has $\clr_{\R^d}(x_0)=+\infty$ and $N_{\mathrm{mesh}}=1$. Consequently, \Cref{alg:continuation} consists of a single call to \Cref{alg:proximal}, and \Cref{thm:certified-continuation} reduces to the whole-space result \Cref{thm:main}.
\end{corollary}

We conclude this section with a unified geometric interpretation of the two feasibility notions and the continuation construction in primal and dual coordinates.

The reciprocal Legendre duality of \Cref{prop:legendre-duality} extends directly to $\X$, which now serves as the primal domain. The bi-Lipschitz estimate \eqref{eq:ordinary-bilip} makes $g$ a homeomorphism onto its image, while invariance of domain shows that this image is open. We therefore define the associated dual domain by $\X^\sharp:=g(\X)$ and the Legendre transform by
\begin{equation*}
E^\sharp(p):=\langle p,g^{-1}(p)\rangle-E(g^{-1}(p)),
\qquad p\in\X^\sharp.
\end{equation*}

As in the whole-space case, $\nabla E^\sharp=g^{-1}$ and $E^\sharp\in\classF_{\mu_\pm^\sharp,L_\pm^\sharp}(\X^\sharp)$, where $\mu_\pm^\sharp=L_\pm^{-1}$ and $L_\pm^\sharp=\mu_\pm^{-1}$. Hence $x$ and $p$ are paired as primal and dual coordinates:
\begin{equation*}
x\in\X
\quad\xleftrightarrow[\;x=g^{-1}(p)=\nabla E^\sharp(p)\;]{\;p=g(x)=\nabla E(x)\;}
\quad p\in\X^\sharp.
\end{equation*}

In these coordinates, the distinction between promised and certified feasibility takes a particularly simple form. Promised feasibility is exactly the pointwise membership statement $0\in\X^\sharp$. Certification at $x_0$ guarantees more: the entire segment
\begin{equation*}
[0,g(x_0)]\subset\X^\sharp,
\qquad
x_\beta^\dagger=\nabla E^\sharp\bigl((1-\beta)g(x_0)\bigr),
\qquad 0\le\beta\le1.
\end{equation*}

Thus certification upgrades the pointwise existence statement $0\in\X^\sharp$ to a usable continuation path: the coordinates $(1-\beta)g(x_0)$ trace a straight segment to the dual origin, while their pullbacks $x_\beta^\dagger$ generate the nested certification balls in the primal domain. Certification thereby supplies the observable localization absent under promised feasibility and restores dimension-free tractability under restricted-domain access.

\section{Conclusion}
\label{sec:conclusion}

This paper studied how much of the global first-order theory of strongly convex and SCSC optimization follows from uniform nondegeneracy itself, without a prescribed curvature orientation. On $\R^d$, the answer is largely affirmative. Signed secant inequalities provide an intrinsic $C^1$ formulation, make the gradient a global bi-Lipschitz homeomorphism, and hence ensure a unique stationary point. Within this geometry, reciprocal Legendre duality and a signed PL inequality re-emerge in forms adapted to the two-sided spectrum.

The central computational result is that the missing scalar descent structure can also be reconstructed. A weighted pair of signed Moreau envelopes produces a merit $V\in C^{1}(\R^d)$ satisfying
\begin{equation*}
V\asymp\norm{g}^2,
\qquad
\norm{\nabla V}\asymp\norm{g}.
\end{equation*}
This simultaneous zeroth- and first-order control makes $V$ both a faithful measure of stationarity and a globally PL objective. Paired proximal descent uses first-order information alone and, from any $x_0\in\R^d$, finds an $\varepsilon$-stationary point in
\begin{equation*}
O\!\left(\frac{L_{\max}^2}{\mu_+\mu_-}\log\frac{\norm{g(x_0)}}\varepsilon\right)
\end{equation*}
queries, independently of the dimension. For strongly convex objectives, the endpoint factor reduces to $L/\mu$, recovering the dependence of standard gradient descent.

The picture changes sharply under restricted-domain oracle access. The same metric geometry, together with a promise that the stationary point lies in the domain, does not ensure that an algorithm can locate it: on the unit ball, the deterministic complexity can be exponential in the dimension at fixed accuracy, even under an infinite-order oracle. Certified feasibility resolves this obstruction by supplying an observable localization of the stationary point. The resulting nested path of certification balls supports feasible continuation and restores dimension-free computation, with an additional continuation term governed by the initial certificate's clearance from the boundary.

Taken together, these results provide a unified account of first-order optimization under uniform nondegeneracy: signed secant geometry supplies the underlying structure, paired proximal descent turns that structure into computation, and certified feasibility identifies the additional information required under restricted-domain access.

\appendix

\section{A specialized global inverse argument}
\label{app:global-inverse}

We verify the Hadamard--L\'evy step used in part~\ref{item:gradient-co-lipschitz} of \Cref{lem:global-gradient-geometry}. Fix $\bar{x}\in\R^d$, set $\bar{p}:=\nabla f(\bar{x})$, and let $\zeta:[0,1]\to\R^d$ be a piecewise $C^1$ path with $\zeta(0)=\bar{p}$. The inverse function theorem gives a unique local lift $\xi$ satisfying
\begin{equation*}
\xi(0)=\bar{x},
\qquad
\nabla f(\xi(t))=\zeta(t).
\end{equation*}
On every smooth piece of $\zeta$, differentiation and \eqref{eq:inverse-hessian-bound} give
\begin{equation}
\dot\xi(t)=\nabla^2f(\xi(t))^{-1}\dot\zeta(t),
\qquad
\norm{\xi(t)-\xi(s)}
\le\nu^{-1}\int_s^t\norm{\dot\zeta(r)}\dd r.
\label{eq:path-lift-bound}
\end{equation}

Suppose the maximal lift were defined only on $[0,T)$ for some $T<1$. The second estimate in \eqref{eq:path-lift-bound} makes $\xi(t)$ Cauchy as $t\uparrow T$, so $\xi(t)\to\xi_T$ for some $\xi_T\in\R^d$. Continuity gives $\nabla f(\xi_T)=\zeta(T)$, and the inverse function theorem at $\xi_T$ extends the lift beyond $T$, a contradiction. Thus every such path has a lift on $[0,1]$. Taking $\zeta$ to be the segment from $\bar{p}$ to an arbitrary $p\in\R^d$ shows that $p=\nabla f(\xi(1))$, so $\nabla f$ is surjective.

To prove injectivity, let $\zeta_0$ and $\zeta_1$ join $\bar{p}$ to the same point $p$. Since $\R^d$ is simply connected, there is a fixed-endpoint homotopy $\zeta_s$, $0\le s\le1$, between them. Let $\xi_s$ denote the lift of $\zeta_s$ from $\bar{x}$. Local invertibility and \eqref{eq:path-lift-bound} make $s\mapsto\xi_s(1)$ continuous, while
\begin{equation*}
\xi_s(1)\in(\nabla f)^{-1}(p),
\qquad 0\le s\le1.
\end{equation*}
The fiber $(\nabla f)^{-1}(p)$ is discrete because $\nabla f$ is a local diffeomorphism, so $\xi_s(1)$ is constant in $s$. Thus the lifted endpoint depends only on $p$. Applying this observation to the image under $\nabla f$ of a path from $\bar{x}$ to any $x\in\R^d$ shows that this endpoint is $x$ when $p=\nabla f(x)$. Hence $\nabla f$ is injective as well as surjective and therefore is a global $C^1$ diffeomorphism.

\section{Topological proof of the index decomposition}
\label{app:index-components}

\begin{proposition}[Index decomposition]
\label[proposition]{prop:index-components}
There is a locally constant map
\begin{equation*}
\iota:\classF_{\mu_\pm,L_\pm}(\R^d)\to\{0,\ldots,d\}.
\end{equation*}
At every differentiability point $x$ of $g$, the matrix $Dg(x)$ is symmetric and
\begin{equation}
\iota(E)=n_-(Dg(x)),
\label{eq:index-ae-inertia}
\end{equation}
where $n_-(M)$ denotes the number of negative eigenvalues, counted with multiplicity. In particular, if $E\in C^2$, then $\iota(E)=n_-(H(x))$ for every $x$, and hence $\iota(E)$ is the Morse index of $x^\star$.

For $k\in\{0,\ldots,d\}$, the subclass $\classF_{\mu_\pm,L_\pm}^{[k]}(\R^d):=\iota^{-1}(k)$ is nonempty and path connected. These $d+1$ strata are precisely the connected components of $\classF_{\mu_\pm,L_\pm}(\R^d)$ in the $C^1_{\mathrm{loc}}$ topology.
\end{proposition}

\begin{proof}
The upper bound in \eqref{eq:ordinary-bilip} makes $g$ Lipschitz. By Rademacher's theorem \cite[Theorem~3.1.2]{EvansGariepy2015}, $g$ is differentiable almost everywhere. Its derivative is symmetric because $g$ is a gradient field and weak mixed derivatives commute. At a differentiability point $x$, write $H_x:=Dg(x)$. Applying \eqref{eq:exact-secant} to $(x+\beta h,x)$, dividing by $\beta^2$, and letting $\beta\downarrow0$ shows that every eigenvalue $\lambda$ of $H_x$ satisfies
\begin{equation*}
(\lambda-\mu_+)(\lambda+\mu_-)\ge0,
\qquad
(\lambda-L_+)(\lambda+L_-)\le0.
\end{equation*}
Thus $\spec(H_x)\subseteq[-L_-,-\mu_-]\cup[\mu_+,L_+]$ at every differentiability point.

For a base point $\bar{x}$, consider the translated gradient map $g_{\bar{x}}(z):=g(\bar{x}+z)-g(\bar{x})$. The potential $E(\bar{x}+z)-E(\bar{x})-\langle g(\bar{x}),z\rangle$ has derivative $-\norm{g_{\bar{x}}(z)}^2$ along $\dot z=-g_{\bar{x}}(z)$, while the lower estimate in \eqref{eq:ordinary-bilip} gives $\norm{g_{\bar{x}}(z)}\ge\mu_{\min}\norm{z}$, so the origin is the unique equilibrium. If a complete trajectory remained in a closed ball centered at the origin, compactness and this Lyapunov identity would force both its $\alpha$- and $\omega$-limit sets to equal $\{0\}$. The potential would then tend to zero in both time directions, which is possible only for the constant trajectory at the origin. Thus the maximal invariant set of every positive-radius closed ball centered at the origin is $\{0\}$, contained in its interior, so each such ball is an isolating neighborhood. Let $\mathfrak{h}(E,\bar{x})$ denote its Conley index \cite[pp.~43--52 and 64--68]{Conley1978}. Varying $\bar{x}$ along a segment gives a Conley continuation with the same isolating neighborhoods, so this index does not depend on the base point; denote its common value by $\mathfrak{h}(E)$.

Now fix a differentiability point $x$. Translations of the argument and affine perturbations preserve the defining secant inequalities and connect $E$ within the class to its normalization at $x$. Extend this normalized member through
\begin{equation*}
E_\beta(z):=
\begin{cases}
\beta^{-2}[E(x+\beta z)-E(x)-\beta\langle g(x),z\rangle],&0<\beta\le1,\\
\frac12z^\top H_xz,&\beta=0.
\end{cases}
\end{equation*}
For $\beta>0$, its gradient is $\beta^{-1}[g(x+\beta z)-g(x)]$, while $\nabla E_0(z)=H_xz$. Fr\'echet differentiability of $g$ at $x$, together with $E_\beta(0)=0$, gives $E_\beta\to E_0$ in $C^1_{\mathrm{loc}}$. Every $E_\beta$ satisfies the original secant inequalities, so the same closed balls isolate the origin throughout this path. Its endpoints generate the translated flow based at $x$ and the linear flow $\dot z=-H_xz$, respectively. Conley continuation therefore gives
\begin{equation*}
\mathfrak{h}(E)=\mathbb{S}^{n_-(H_x)}.
\end{equation*}
Thus $\iota(E):=n_-(H_x)$ is independent of $x$. Robustness of the Conley index under locally uniform perturbations makes $\iota$ locally constant in the $C^1_{\mathrm{loc}}$ topology and proves \eqref{eq:index-ae-inertia}.

The same path joins every member of the $k$th stratum to the quadratic $z\mapsto\frac12z^\top H_xz$. Deforming its positive eigenvalues to $\mu_+$ and its negative eigenvalues to $-\mu_-$ preserves the spectral intervals. The negative eigenspace can then be rotated continuously to the first $k$ coordinate directions. Hence every member is joined to
\begin{equation*}
Q_k(z):=\frac12z^\top\diag(-\mu_-I_k,\mu_+I_{d-k})z.
\end{equation*}
These quadratics show that every stratum is nonempty; the constructed paths make each stratum path connected, while local constancy separates distinct strata. This proves the component decomposition in \eqref{eq:index-decomposition}.
\end{proof}

\section{Block Hessian spectral bounds}
\label{app:classical-subclasses}

Let $A\in\R^{d_x\times d_x}$ and $C\in\R^{d_y\times d_y}$ be symmetric, let $B\in\R^{d_x\times d_y}$, and suppose
\begin{equation*}
\mu_xI\preceq A\preceq L_xI,\qquad
\mu_yI\preceq C\preceq L_yI,\qquad
\norm{B}\le L_{xy}.
\end{equation*}
We verify the spectral claim used in the proof of \Cref{prop:classical-subclasses} for
\begin{equation*}
H:=\begin{pmatrix}A&B\\B^\top&-C\end{pmatrix}.
\end{equation*}

Let $\lambda>0$ be an eigenvalue of $H$ with eigenvector $(\xi,\eta)$. Necessarily $\xi\ne0$, and eliminating the second block gives $\eta=(C+\lambda I)^{-1}B^\top\xi$ and hence
\begin{equation*}
\lambda\norm{\xi}^2
=\langle A\xi,\xi\rangle
+\langle B(C+\lambda I)^{-1}B^\top\xi,\xi\rangle
\ge\mu_x\norm{\xi}^2.
\end{equation*}
Thus $\lambda\ge\mu_x$. Similarly, if $-\lambda<0$ is an eigenvalue, then $\eta\ne0$, $\xi=-(A+\lambda I)^{-1}B\eta$, and
\begin{equation*}
\lambda\norm{\eta}^2
=\langle C\eta,\eta\rangle
+\langle B^\top(A+\lambda I)^{-1}B\eta,\eta\rangle
\ge\mu_y\norm{\eta}^2,
\end{equation*}
so $\lambda\ge\mu_y$.

For the outer endpoints, the Rayleigh quotient satisfies
\begin{equation*}
\begin{aligned}
\mu_x\norm{\xi}^2-2L_{xy}\norm{\xi}\norm{\eta}-L_y\norm{\eta}^2
&\le\langle H(\xi,\eta),(\xi,\eta)\rangle\\
&\le L_x\norm{\xi}^2+2L_{xy}\norm{\xi}\norm{\eta}-\mu_y\norm{\eta}^2.
\end{aligned}
\end{equation*}
The lower scalar quadratic form has smallest eigenvalue $-L_-$, and the upper one has largest eigenvalue $L_+$, with $L_\pm$ defined in \eqref{eq:scsc-outer-endpoints}. Therefore
\begin{equation*}
\spec(H)\subseteq[-L_-,-\mu_y]\cup[\mu_x,L_+].
\end{equation*}
Finally, $A\succ0$ and its Schur complement $-C-B^\top A^{-1}B\prec0$. Sylvester's law of inertia shows that $H$ has exactly $d_y$ negative eigenvalues.

\section{A deterministic first-order lower bound on the whole space}
\label{app:whole-space-lower-bound}

This appendix establishes the deterministic lower benchmark in \Cref{rem:whole-space-lower-bound}. We first treat two nondegenerate spectral intervals through a four-endpoint estimate and then cover a collapsed branch by a one-interval reduction.

Define
\begin{equation*}
\kappa_\pm:=\frac{L_\pm}{\mu_\pm},
\qquad
\psi(z):=\log\frac{z+1}{z-1}\quad(z>1).
\end{equation*}
When $\kappa_\pm>1$, define the endpoint cross-ratio
\begin{equation*}
\chi:=\frac{(L_-+\mu_+)(L_++\mu_-)}{(L_--\mu_-)(L_+-\mu_+)}.
\end{equation*}
The endpoint cross-ratio $\chi>1$ is invariant under a common scaling of the four endpoints and under reflection of the two spectral branches.

We count one query for each returned pair $(E(x),g(x))$, including the response at the prescribed initial point $x_0$; the final output need not have been queried.

\begin{theorem}[Endpoint-sensitive deterministic lower bound]
\label[theorem]{thm:endpoint-lower-bound}
Suppose $\kappa_\pm>1$. Fix an initial point $x_0$ and prescribe $\norm{g(x_0)}>0$. For $0<\varepsilon\le\norm{g(x_0)}/2$, suppose that
\begin{equation}
d\ge
2\left\lceil
\frac{\log(16\chi)}{\psi(\sqrt{\kappa_+})\psi(\sqrt{\kappa_-})}
\log\frac{\norm{g(x_0)}}{\varepsilon}
\right\rceil+2,
\label{eq:endpoint-lower-dimension}
\end{equation}
then every deterministic first-order method that returns an $\varepsilon$-stationary point for every $E\in\classF_{\mu_\pm,L_\pm}(\R^d)\cap C^\infty(\R^d)$ with this initial-residual magnitude has worst-case query count
\begin{equation}
Q\ge
\frac{\log(16\chi)}{\psi(\sqrt{\kappa_+})\psi(\sqrt{\kappa_-})}
\log\frac{\norm{g(x_0)}}{\varepsilon}\ge
\frac{1}{\psi(\sqrt{\kappa_+\kappa_-})}
\log\frac{\norm{g(x_0)}}{\varepsilon}.
\label{eq:endpoint-sensitive-lower-bound}
\end{equation}
\end{theorem}

We prove the two inequalities in \eqref{eq:endpoint-sensitive-lower-bound} in sequence. The first reduction comes from Nemirovsky's optimality theorem: on quadratic instances, a first-order transcript is matrix--vector information, and the best attainable residual is governed by polynomial approximation on the spectrum. To formulate this reduction, let
\begin{equation*}
\Sigma:=[-L_-,-\mu_-]\cup[\mu_+,L_+]
\end{equation*}
and, for $n\ge0$, define the residual polynomial value
\begin{equation*}
\varrho_n(\Sigma):=
\inf_{\substack{p\in\R[\lambda],\ \deg p\le n\\p(0)=1}}
\norm{p}_{\infty,\Sigma}.
\end{equation*}
The following lemma makes this reduction precise.

\begin{lemma}[Deterministic information reduction]
\label[lemma]{lem:deterministic-information-reduction}
Let $n\ge1$ and $d\ge2n+2$. For every deterministic method using at most $n$ first-order queries, and for every prescribed positive value of $\norm{g(x_0)}$, there exists a quadratic $E\in\classF_{\mu_\pm,L_\pm}(\R^d)$ for which
\begin{equation}
\frac{\norm{g(x_n)}}{\norm{g(x_0)}}
\ge\varrho_n(\Sigma).
\label{eq:deterministic-information-reduction}
\end{equation}
\end{lemma}

\begin{proof}
Translate $x_0$ to the origin and consider the normalized quadratics
\begin{equation*}
E(x)=\frac12x^\top Ax+b^\top x,
\qquad
g(x)=Ax+b,
\qquad
\spec(A)\subseteq\Sigma.
\end{equation*}
The initial response reveals $b$, and each later response reveals one adaptive matrix--vector product $Ax=g(x)-b$. The accompanying function value is redundant: $E(x)=\frac12\langle x,g(x)+b\rangle$ is already determined by $x$, $g(x)$, and $b$. Thus $n$ first-order queries supply at most $n-1$ adaptive products with $A$.

Nemirovsky's optimality theorem for signed-spectrum symmetric equations \cite[Theorem, p.~124]{Nemirovsky1991}, with the extension noted in \cite[p.~157]{Nemirovsky1992}, applies to $Ax=-b$ with accuracy measured by the relative residual $\norm{Ax+b}/\norm{b}$. After $k=n-1$ matrix--vector products, it shows that no method can guarantee a relative residual below $\varrho_{k+1}(\Sigma)=\varrho_n(\Sigma)$. The theorem permits an arbitrary transcript-dependent output, and its dimension condition $2k\le d-3$ follows from $d\ge2n+2$. Scaling $b$ gives any prescribed value of $\norm{g(x_0)}$, and translating back proves \eqref{eq:deterministic-information-reduction}.
\end{proof}

With the oracle reduction in hand, classical potential theory converts the polynomial problem into a scalar bound. Let $G$ be the Green function of the unbounded domain $\widehat{\mathbb{C}}\setminus\Sigma$ with pole at infinity. The Bernstein--Walsh inequality at $0\notin\Sigma$ \cite[Theorem~5.5.7(a)]{Ransford1995} gives
\begin{equation*}
\varrho_n(\Sigma)\ge e^{-nG(0)}.
\end{equation*}
Consequently, if $d\ge2Q+2$, no deterministic method using $Q$ responses can guarantee residual at most $\varepsilon$ unless
\begin{equation}
Q\ge\frac1{G(0)}\log\frac{\norm{g(x_0)}}{\varepsilon}.
\label{eq:green-query-lower-bound}
\end{equation}
The first inequality of the theorem will therefore follow once $G(0)$ is bounded from the four endpoints. For this purpose, on the spectral gap $-\mu_-<\lambda<\mu_+$, set
\begin{equation*}
w(\lambda):=
\frac1{\sqrt{(\lambda+L_-)(\lambda+\mu_-)(\mu_+-\lambda)(L_+-\lambda)}}
\end{equation*}
and define
\begin{equation*}
I_-:=\int_{-\mu_-}^0w(\lambda)\dd\lambda,
\qquad
I_+:=\int_0^{\mu_+}w(\lambda)\dd\lambda,
\end{equation*}
\begin{equation*}
J:=\int_{-\mu_-}^0\int_0^{\mu_+}
(\lambda_+-\lambda_-)w(\lambda_-)w(\lambda_+)\dd\lambda_+\dd\lambda_-.
\end{equation*}
The classical two-interval Green representation of Shen, Strang, and Wathen \cite[Sec.~2.1]{ShenStrangWathen2001}, after eliminating its weighted gap critical point, gives the reflection-symmetric identity
\begin{equation}
G(0)=\frac{J}{I_-+I_+}.
\label{eq:two-interval-green-value}
\end{equation}
The following elementary inequality is the only comparison needed inside this representation.

\begin{lemma}[Ordered-pair separation]
\label[lemma]{lem:ordered-pair-separation}
Suppose
\begin{equation*}
0\le X_\ell\le X_c\le X_r,
\qquad
0\le Y_\ell\le Y_c\le Y_r,
\qquad
X_c,Y_c>0.
\end{equation*}
Then
\begin{equation}
\frac{X_\ell+Y_\ell}
{\sqrt{(X_\ell+Y_c)(X_\ell+Y_r)(Y_\ell+X_c)(Y_\ell+X_r)}}
\le
\frac1{\sqrt{(X_r+Y_c)(Y_r+X_c)}}.
\label{eq:ordered-pair-separation}
\end{equation}
\end{lemma}

\begin{proof}
The ordering gives
\begin{equation*}
\frac{X_r+Y_c}{X_r+Y_\ell}
\le\frac{X_c+Y_c}{X_c+Y_\ell},
\qquad
\frac{Y_r+X_c}{Y_r+X_\ell}
\le\frac{Y_c+X_c}{Y_c+X_\ell},
\end{equation*}
while
\begin{equation*}
(X_\ell+Y_c)(Y_\ell+X_c)-(X_c+Y_c)(X_\ell+Y_\ell)
=(X_c-X_\ell)(Y_c-Y_\ell)\ge0.
\end{equation*}
Consequently,
\begin{equation*}
\frac{(X_\ell+Y_\ell)^2(X_r+Y_c)(Y_r+X_c)}
{(X_\ell+Y_c)(X_\ell+Y_r)(Y_\ell+X_c)(Y_\ell+X_r)}
\le
\left(
\frac{(X_c+Y_c)(X_\ell+Y_\ell)}
{(X_\ell+Y_c)(Y_\ell+X_c)}
\right)^2
\le1.
\end{equation*}
Taking square roots gives \eqref{eq:ordered-pair-separation}.
\end{proof}

Returning to \eqref{eq:two-interval-green-value}, we first evaluate its denominator. Let
\begin{equation*}
\mathbf{K}(k):=\int_0^{\pi/2}\frac{\dd\theta}{\sqrt{1-k^2\sin^2\theta}},
\qquad 0\le k<1,
\end{equation*}
denote the complete elliptic integral of the first kind \cite[Eqs.~19.2.4 and 19.2.8]{NISTDLMF}. Standard quartic reduction formulas \cite[Eqs.~19.25.1 and 19.29.4]{NISTDLMF} give
\begin{equation}
I_-+I_+
=\frac{2\mathbf{K}(\sqrt{1-\chi^{-1}})}{\sqrt{(L_-+\mu_+)(L_++\mu_-)}}.
\label{eq:elliptic-gap-mass}
\end{equation}
For the numerator, we parameterize $\lambda_-= -\mu_-(1-\xi_-^2)$ and $\lambda_+=\mu_+(1-\xi_+^2)$ and apply \Cref{lem:ordered-pair-separation} with
\begin{equation*}
(X_\ell,X_c,X_r)=(-\lambda_-,\mu_-,L_-),
\qquad
(Y_\ell,Y_c,Y_r)=(\lambda_+,\mu_+,L_+).
\end{equation*}
This gives
\begin{align*}
J
&\le
\frac4{\sqrt{(L_-+\mu_+)(L_++\mu_-)}}
\left(\int_0^1\frac{\dd\xi_-}{\sqrt{\xi_-^2+\kappa_--1}}\right)
\left(\int_0^1\frac{\dd\xi_+}{\sqrt{\xi_+^2+\kappa_+-1}}\right)\\
&=
\frac{\psi(\sqrt{\kappa_+})\psi(\sqrt{\kappa_-})}
{\sqrt{(L_-+\mu_+)(L_++\mu_-)}}.
\end{align*}
Combining this estimate with \eqref{eq:two-interval-green-value} and \eqref{eq:elliptic-gap-mass} yields
\begin{equation*}
\frac1{G(0)}
\ge
\frac{2\mathbf{K}(\sqrt{1-\chi^{-1}})}
{\psi(\sqrt{\kappa_+})\psi(\sqrt{\kappa_-})}.
\end{equation*}
The classical estimate $2\mathbf{K}(\sqrt{1-\chi^{-1}})\ge\log(16\chi)$ \cite[Eq.~19.9.1]{NISTDLMF} now gives
\begin{equation}
\frac1{G(0)}
\ge
\frac{\log(16\chi)}
{\psi(\sqrt{\kappa_+})\psi(\sqrt{\kappa_-})}.
\label{eq:explicit-green-lower-bound}
\end{equation}

We can now return to the oracle problem. Suppose that an integer $Q$ violates the first inequality in \eqref{eq:endpoint-sensitive-lower-bound}. Then \eqref{eq:endpoint-lower-dimension} implies $d\ge2Q+2$, whereas \eqref{eq:green-query-lower-bound} and \eqref{eq:explicit-green-lower-bound} give the opposite conclusion. This proves the first inequality of \eqref{eq:endpoint-sensitive-lower-bound}.

To obtain the second inequality, we now eliminate $\chi$ from the coefficient. The following elementary comparison does exactly this.

\begin{lemma}[Product reduction]
\label[lemma]{lem:product-reduction}
The quantities defined above satisfy
\begin{equation}
\frac{\log(16\chi)}
{\psi(\sqrt{\kappa_+})\psi(\sqrt{\kappa_-})}
\ge
\frac1{\psi(\sqrt{\kappa_+\kappa_-})}.
\label{eq:product-reduction}
\end{equation}
\end{lemma}

\begin{proof}
Set
\begin{equation*}
\psi_\pm:=\psi(\sqrt{\kappa_\pm}),
\qquad
\bar{\psi}:=\frac{\psi_++\psi_-}{2}.
\end{equation*}
Since $\sqrt{\kappa_\pm}=\coth(\psi_\pm/2)$,
\begin{equation*}
\psi(\sqrt{\kappa_+\kappa_-})
=\log\frac{\cosh\bar{\psi}}{\cosh(|\psi_+-\psi_-|/2)}
=\int_{|\psi_+-\psi_-|/2}^{\bar{\psi}}\tanh u\dd u.
\end{equation*}
Since $\tanh 0=0$, concavity of $\tanh$ on $[0,\infty)$ gives
\begin{equation*}
\tanh u\ge\frac{u}{\bar{\psi}}\tanh\bar{\psi},
\qquad 0\le u\le\bar{\psi}.
\end{equation*}
Therefore,
\begin{equation*}
\psi(\sqrt{\kappa_+\kappa_-})
\ge\frac{\tanh\bar{\psi}}{\bar{\psi}}
\int_{|\psi_+-\psi_-|/2}^{\bar{\psi}}u\dd u
=\frac{\psi_+\psi_-}{2\bar{\psi}}\tanh\bar{\psi}.
\end{equation*}
Consequently, taking reciprocals and using $\sinh\bar{\psi}\ge\bar{\psi}$ gives
\begin{equation*}
\frac{\psi_+\psi_-}{\psi(\sqrt{\kappa_+\kappa_-})}
\le\frac{2\bar{\psi}}{\tanh\bar{\psi}}
\le2\sqrt{1+\bar{\psi}^2}.
\end{equation*}
Direct substitution in the definition of $\chi$ gives
\begin{equation*}
\chi=
\left(\cosh^2\frac{\psi_-}{2}+\frac{\mu_+}{\mu_-}\sinh^2\frac{\psi_-}{2}\right)
\left(\cosh^2\frac{\psi_+}{2}+\frac{\mu_-}{\mu_+}\sinh^2\frac{\psi_+}{2}\right)
\ge\cosh^2\bar{\psi},
\end{equation*}
where the inequality follows from AM--GM. Moreover, the function
\begin{equation*}
u\longmapsto\log16+2\log\cosh u-2\sqrt{1+u^2}
\end{equation*}
is positive at $u=0$ and nondecreasing at $u>0$. 
Hence
\begin{equation*}
\log(16\chi)\ge\log16+2\log\cosh\bar{\psi}
\ge2\sqrt{1+\bar{\psi}^2}
\ge\frac{\psi_+\psi_-}{\psi(\sqrt{\kappa_+\kappa_-})},
\end{equation*}
which is equivalent to \eqref{eq:product-reduction}.
\end{proof}

Combining the first inequality with \Cref{lem:product-reduction} proves the second and completes \Cref{thm:endpoint-lower-bound}.

It remains only to cover a collapsed spectral branch. If $\kappa_-=1<\kappa_+$, restrict the quadratic family in \Cref{lem:deterministic-information-reduction} to matrices with spectrum in $[\mu_+,L_+]$. The corresponding one-interval Green value at zero is $\psi(\sqrt{\kappa_+})$. Hence, whenever
\begin{equation*}
d\ge2\left\lceil
\frac1{\psi(\sqrt{\kappa_+})}
\log\frac{\norm{g(x_0)}}\varepsilon
\right\rceil+2,
\end{equation*}
the same reduction gives
\begin{equation*}
Q\ge
\frac1{\psi(\sqrt{\kappa_+})}
\log\frac{\norm{g(x_0)}}\varepsilon
=
\frac1{\psi(\sqrt{\kappa_+\kappa_-})}
\log\frac{\norm{g(x_0)}}\varepsilon.
\end{equation*}
The case $\kappa_+=1<\kappa_-$ follows by sign reversal. If $\kappa_+=\kappa_-=1$, the convention $\psi(1):=+\infty$ makes the product-form lower bound trivial. Thus \eqref{eq:whole-space-lower-benchmark} holds for all admissible endpoint configurations.

\section{Spatial confinement of paired proximal descent}
\label{app:spatial-confinement}

In this appendix, we prove \Cref{lem:spatial-confinement} by a first-exit argument. The proof has two stages: we first reconstruct locally the paired proximal geometry of \Cref{lem:proximal-path,thm:paired-proximal-equivalence}, and then use the inner contraction \eqref{eq:inner-contraction} and merit descent \eqref{eq:merit-descent} to confine the complete oracle trajectory.

The first stage is necessary because the whole-space variational construction cannot simply be restricted to $\X$. In \Cref{lem:proximal-path}, the proximal point $y_s(x)$ is obtained as the maximizer of an objective whose strong-concavity bound on $\R^d$ gives an attained unique maximum. On a general open domain $\X$, however, convexity may fail, so the variational strong-concavity argument no longer provides uniqueness, while openness leaves attainment unsecured. We must therefore reverse the construction. Inside a convex subdomain $\mathcal{U}\subset\X$, we first construct $y_s(x)$ as the unique solution of the proximal equation $s\nabla h(y)=y-x$, using strong monotonicity for localization and uniqueness and Brouwer degree for existence \cite[Chap.~1]{Deimling1985}. Only then do we recover its variational characterization and use it to define the local envelopes and merit. The following lemma reconstructs precisely the proximal and merit properties needed by the \PPD{} analysis of \Cref{sec:proximal}.

Recalling $m_\pm$ from \eqref{eq:branch-curvature}, define $m_{\min}:=\min\{m_+,m_-\}$ and $\mu_{\mathrm{loc}}:=m_{\min}\mu_{\min}$.

\begin{lemma}[Local reconstruction of paired proximal geometry]
\label[lemma]{lem:local-paired-geometry}
Let $h\in\classF_{\mu_\pm,L_\pm}(\X)$ have stationary point $x^\dagger$, and let $\mathcal{U}\subset\X$ be an open convex set containing $x^\dagger$. Call an anchor $x\in\mathcal{U}$ safe relative to $\mathcal{U}$ if
\begin{equation}
\overline{B}\!\left(x,\frac2{\mu_{\mathrm{loc}}}\norm{\nabla h(x)}\right)\subset\mathcal{U}.
\label{eq:safe-anchor}
\end{equation}
For every safe anchor $x$ and every $s\in[-1/\Lambda_-,1/\Lambda_+]$, the proximal equation
\begin{equation}
s\nabla h(y_s(x))=y_s(x)-x
\label{eq:local-proximal-equation}
\end{equation}
has a unique solution in $\X$, which lies in $\mathcal{U}$ and satisfies
\begin{equation}
\norm{y_s(x)-x}\le\frac1{\mu_{\mathrm{loc}}}\norm{\nabla h(x)},
\qquad
\norm{\nabla h(y_s(x))}\le\frac1{m_{\min}}\norm{\nabla h(x)}.
\label{eq:local-proximal-bounds}
\end{equation}
The solution is the unique maximizer of the corresponding signed proximal objective over $\mathcal{U}$, so
\begin{equation*}
\begin{aligned}
\mathcal{M}_s(x)
&:=s\bigl[h(y_s(x))-h(x^\dagger)\bigr]-\frac12\norm{y_s(x)-x}^2\\
&=\max_{y\in\mathcal{U}}
\left\{
s\bigl[h(y)-h(x^\dagger)\bigr]-\frac12\norm{y-x}^2
\right\}.
\end{aligned}
\end{equation*}
The safe anchors form an open set. With the endpoint notation \eqref{eq:proximal-endpoints}, these local envelopes define $V$ and $v$ by \eqref{eq:proximal-merit} and \eqref{eq:paired-merit-gradient}. On the safe-anchor set, $V\in C^1$, $\nabla V=v$, and the merit equivalences \eqref{eq:merit-value-equivalence}--\eqref{eq:merit-gradient-equivalence}, merit-smoothness estimate \eqref{eq:merit-gradient-lipschitz}, and proximal-target estimate \eqref{eq:proximal-target-lipschitz} hold with their Section~3 constants. The one-step specialization of \eqref{eq:inner-contraction} also holds whenever the current inner iterate lies in $\mathcal{U}$.
\end{lemma}

\begin{proof}
Fix an anchor $x$ satisfying \eqref{eq:safe-anchor} and $s\in[-1/\Lambda_-,1/\Lambda_+]$, and retain the branch modulus $m_s$ from \eqref{eq:proximal-path-modulus}, for which $m_s\ge m_{\min}>0$. For $0\le\alpha\le1$, consider the homotopy of equations generated by
\begin{equation*}
F_{\alpha,s}(y):=y-\alpha s\nabla h(y)-x.
\end{equation*}
This provides a homotopy from the equation with known root $y=x$ to the target equation defining $y_s(x)$. At $\alpha=0$, the map $F_{0,s}(y)=y-x$ has the unique zero $x$, whereas $F_{1,s}(y)=0$ is exactly \eqref{eq:local-proximal-equation}.

To keep the degree well defined along the homotopy, we first localize all of its zeros. The ratio $|s|/m_s$ is maximized at the two branch endpoints, where $\Lambda_\pm m_\pm\ge\mu_{\min}m_{\min}=\mu_{\mathrm{loc}}$. Since $\Id-\alpha s\nabla h$ is strongly monotone with modulus at least $m_s$, every zero of $F_{\alpha,s}$ satisfies
\begin{equation*}
\norm{y-x}\le\frac{\alpha|s|}{m_s}\norm{\nabla h(x)}\le\max\left\{\frac1{\Lambda_+-L_+},\frac1{\Lambda_--L_-}\right\}\norm{\nabla h(x)}\le\frac1{\mu_{\mathrm{loc}}}\norm{\nabla h(x)}.
\end{equation*}

This localization makes the degree argument possible. If $\nabla h(x)=0$, strong monotonicity makes $x$ the unique zero. Otherwise, consider the open ball $B\bigl(x,3\norm{\nabla h(x)}/(2\mu_{\mathrm{loc}})\bigr)$. Its closure lies in $\mathcal{U}$, and the localization bound excludes zeros of $F_{\alpha,s}$ from its boundary. Since $F_{0,s}(y)=y-x$ has degree $1$ on this ball at zero, homotopy invariance of the Brouwer degree \cite[Chap.~1]{Deimling1985} gives a zero of $F_{1,s}$. Strong monotonicity makes it unique throughout $\X$, and the localization bound proves the first estimate in \eqref{eq:local-proximal-bounds}.

Existence of the root now recovers the variational formulation rather than following from it. The signed proximal objective is $m_s$-strongly concave on the convex set $\mathcal{U}$, and \eqref{eq:local-proximal-equation} makes $y_s$ stationary; hence $y_s$ is its attained unique maximizer. The pairwise proof of \eqref{eq:proximal-path-residual} applies to $x,y_s\in\mathcal{U}$ and gives $m_s\norm{\nabla h(y_s)}\le\norm{\nabla h(x)}$, proving the second estimate in \eqref{eq:local-proximal-bounds}. Consequently, \eqref{eq:root-certification-ball} yields
\begin{equation*}
\Bcert(y_s)
\subset
\overline{B}\!\left(y_s,\frac{\norm{\nabla h(y_s)}}{\mu_{\min}}\right)
\subset
\overline{B}\!\left(x,\frac2{\mu_{\mathrm{loc}}}\norm{\nabla h(x)}\right)
\subset\mathcal{U}.
\end{equation*}

The homotopy in the proof of \Cref{prop:certification-implies-promise}, localized by \eqref{eq:certified-homotopy-localization}, now supplies $z_t\in\mathcal{U}$ with $\nabla h(z_t)=t\nabla h(y_s)$, $z_0=x^\dagger$, and $z_1=y_s$. Uniqueness and \eqref{eq:ordinary-bilip} make $t\mapsto z_t$ Lipschitz, while the lower secant inequality gives the local form of \eqref{eq:inverse-sector}. The chain rule along this path yields
\begin{equation*}
h(y_s)-h(x^\dagger)=\int_0^1\left\langle y_s-z_t,\nabla h(y_s)\right\rangle\dd t,
\end{equation*}
so the remaining one-dimensional argument in the proof of \Cref{prop:signed-pl} gives \eqref{eq:signed-pl} at $y_s$.

Pairwise strong monotonicity makes $y_s$ Lipschitz in $x$ and gives $\norm{y_{s_1}(x)-y_{s_2}(x)}\le |s_1-s_2|\norm{\nabla h(x)}/m_{\min}^2$. Since safety persists under small perturbations, nearby maximizers remain in a common compact subset of $\mathcal{U}$. The envelope differentiation from \Cref{lem:proximal-path} therefore applies locally and gives $\nabla\mathcal{M}_s=y_s-\Id$ and $\partial_s\mathcal{M}_s=h(y_s)-h(x^\dagger)$. The local residual comparison and signed PL estimate are now precisely the inputs used in \eqref{eq:merit-integral-representation} to prove \eqref{eq:merit-value-equivalence}; the endpoint equations and the pairwise calculation \eqref{eq:merit-gradient-cross-secant}--\eqref{eq:proximal-endpoint-separation} similarly give \eqref{eq:merit-gradient-equivalence}.

The same pairwise arguments prove \eqref{eq:proximal-target-lipschitz} and \eqref{eq:merit-gradient-lipschitz}. The one-step specialization of \eqref{eq:inner-contraction} also remains valid while the current inner iterate lies in $\mathcal{U}$.
\end{proof}

The local lemma has recovered every analytic object needed by \PPD; it remains to keep the algorithm in the region where those objects are available. Merit descent and value equivalence \eqref{eq:merit-descent} and \eqref{eq:merit-value-equivalence} bound the anchor residuals, the lower metric estimate \eqref{eq:ordinary-bilip} confines the anchors around $x^\dagger$, and the inner-iteration contraction \eqref{eq:inner-contraction} confines the cold and warm inner trajectories. Set $K_0:=\overline{C}_0/\underline{C}_0\ge1$ and $\delta_{\max}:=\max\{\delta_+,\delta_-\}$, and define
\begin{equation*}
\Gamma
:=
\frac{8\sqrt{K_0}(1+\delta_{\max})}{\mu_{\mathrm{loc}}}
\left(1+\frac{L_{\max}}{\mu_{\mathrm{loc}}}\right).
\end{equation*}
This choice is uniform over all admissible values of $\tau$.

\begingroup
\renewcommand{\proofname}{Proof of \Cref{lem:spatial-confinement}}
\begin{proof}
If $\nabla h(x_{\mathrm{in}})=0$, then the lower metric estimate \eqref{eq:ordinary-bilip} gives $x_{\mathrm{in}}=x^\dagger$, and the algorithm returns immediately. Assume henceforth that $\nabla h(x_{\mathrm{in}})\ne0$ and consider the open ball
\begin{equation*}
\mathcal{U}
:=
B\!\left(
x^\dagger,
\Gamma\norm{\nabla h(x_{\mathrm{in}})}
\right).
\end{equation*}
The hypothesis \eqref{eq:certification-confinement-condition} gives $\overline{\mathcal{U}}\subset\X$. Let $V$ be the local merit on the safe-anchor set supplied by \Cref{lem:local-paired-geometry}.

We prove confinement inductively over the completed inner phases. Our induction hypothesis at the end of each phase consists of four conditions: (i) every oracle query made so far lies in $\mathcal{U}$; (ii) the current anchor satisfies \eqref{eq:safe-anchor}; (iii) its merit satisfies $V(x)\le V(x_{\mathrm{in}})$; and (iv) its proximal approximations satisfy \eqref{eq:relative-proximal}. The cold phase verifies the base case, while each outer update and the ensuing warm phase establish the same conditions at the next anchor.

We first verify these conditions at the cold anchor. By \eqref{eq:ordinary-bilip}, $\norm{x_{\mathrm{in}}-x^\dagger}\le\norm{\nabla h(x_{\mathrm{in}})}/\mu_{\mathrm{loc}}$. Hence
\begin{equation*}
\overline{B}\!\left(x_{\mathrm{in}},\frac2{\mu_{\mathrm{loc}}}\norm{\nabla h(x_{\mathrm{in}})}\right)
\subset
\overline{B}\!\left(x^\dagger,\frac3{\mu_{\mathrm{loc}}}\norm{\nabla h(x_{\mathrm{in}})}\right)
\subset\mathcal{U}.
\end{equation*}
Thus $x_{\mathrm{in}}$ is safe. By \eqref{eq:local-proximal-bounds}, its exact targets are within radius $2\norm{\nabla h(x_{\mathrm{in}})}/\mu_{\mathrm{loc}}$ of $x^\dagger$. Starting the cold iteration at $x_{\mathrm{in}}$, the one-step specialization of \eqref{eq:inner-contraction} gives
\begin{equation*}
\norm{\widehat{y}_\pm^{(n+1)}-x^\dagger}
\le\frac3{\mu_{\mathrm{loc}}}\norm{\nabla h(x_{\mathrm{in}})}
<\Gamma\norm{\nabla h(x_{\mathrm{in}})}.
\end{equation*}
This bound places each successor back in $\mathcal{U}$, so induction confines all cold inner iterates $\widehat{y}_\pm^{(n)}$ to $\mathcal{U}$. The cold-start argument of \Cref{prop:proximal-tracking} and the budgets \eqref{eq:tracking-budgets} now establish \eqref{eq:relative-proximal}, completing the base case.

Suppose the four induction conditions hold at a current anchor $x$. The local value equivalence \eqref{eq:merit-value-equivalence} gives
\begin{equation}
\norm{\nabla h(x)}
\le
\sqrt{K_0}\norm{\nabla h(x_{\mathrm{in}})},
\qquad
\norm{x-x^\dagger}
\le
\frac{\sqrt{K_0}}{\mu_{\mathrm{loc}}}
\norm{\nabla h(x_{\mathrm{in}})}.
\label{eq:confined-anchor-bounds}
\end{equation}
Indeed, the first bound follows from $\underline{C}_0\norm{\nabla h(x)}^2\le V(x)\le V(x_{\mathrm{in}})\le\overline{C}_0\norm{\nabla h(x_{\mathrm{in}})}^2$, and the second from \eqref{eq:ordinary-bilip}. Moreover, \eqref{eq:local-proximal-bounds} and \eqref{eq:paired-merit-gradient} give $\norm{v(x)}\le\norm{\nabla h(x)}/\mu_{\mathrm{loc}}$. Using \eqref{eq:relative-field}, $L_v\ge1$, and $\tau(1+\delta)<1$, the outer update therefore satisfies
\begin{equation*}
\norm{x_{\mathrm{new}}-x}
\le\frac1{\mu_{\mathrm{loc}}}\norm{\nabla h(x)}
\le\frac{\sqrt{K_0}}{\mu_{\mathrm{loc}}}\norm{\nabla h(x_{\mathrm{in}})}.
\end{equation*}

These bounds first place $x_{\mathrm{new}}$, and hence the segment $[x,x_{\mathrm{new}}]$, inside the convex set $\mathcal{U}$. For any $z$ on this segment, the triangle inequality gives $\norm{z-x^\dagger}\le2\sqrt{K_0}\norm{\nabla h(x_{\mathrm{in}})}/\mu_{\mathrm{loc}}$, while the upper metric estimate gives $\norm{\nabla h(z)}\le\sqrt{K_0}(1+L_{\max}/\mu_{\mathrm{loc}})\norm{\nabla h(x_{\mathrm{in}})}$. Recentering the safe ball at $x^\dagger$ therefore gives
\begin{equation*}
\begin{aligned}
\norm{z-x^\dagger}+\frac2{\mu_{\mathrm{loc}}}\norm{\nabla h(z)}
&\le\frac{2\sqrt{K_0}}{\mu_{\mathrm{loc}}}\left(2+\frac{L_{\max}}{\mu_{\mathrm{loc}}}\right)\norm{\nabla h(x_{\mathrm{in}})}
\le\frac\Gamma2\norm{\nabla h(x_{\mathrm{in}})},\\
\overline{B}\!\left(z,\frac2{\mu_{\mathrm{loc}}}\norm{\nabla h(z)}\right)
&\subset\overline{B}\!\left(x^\dagger,\frac\Gamma2\norm{\nabla h(x_{\mathrm{in}})}\right)\subset\mathcal{U}.
\end{aligned}
\end{equation*}
Thus every point of the outer segment is safe. The proof of \Cref{lem:inexact-merit-descent} applies on this segment, so $V(x_{\mathrm{new}})\le V(x)\le V(x_{\mathrm{in}})$. In particular, $x_{\mathrm{new}}$ is safe and satisfies \eqref{eq:confined-anchor-bounds}.

If the stopping test fails there, it remains to confine the warm refinement. The warm-start estimate from the proof of \Cref{prop:proximal-tracking}, together with \eqref{eq:confined-anchor-bounds} and $m_{\min}^{-1}\le L_{\max}/\mu_{\mathrm{loc}}$, gives
\begin{equation*}
\begin{aligned}
\norm{\widehat{y}_\pm^{(0)}-y_\pm(x_{\mathrm{new}})}
&\le\left(\delta_\pm+\frac{(1+\delta)\tau}{m_\pm L_v}\right)\norm{v(x)}\\
&\le\frac{\sqrt{K_0}(1+\delta_{\max})}{\mu_{\mathrm{loc}}}
\left(1+\frac{L_{\max}}{\mu_{\mathrm{loc}}}\right)
\norm{\nabla h(x_{\mathrm{in}})}.
\end{aligned}
\end{equation*}
The local proximal bound and \eqref{eq:confined-anchor-bounds} place $y_\pm(x_{\mathrm{new}})$ within $2\sqrt{K_0}\norm{\nabla h(x_{\mathrm{in}})}/\mu_{\mathrm{loc}}$ of $x^\dagger$. Starting from the retained approximation, the inner-iteration contraction therefore keeps every warm query within
\begin{equation*}
\frac{\sqrt{K_0}}{\mu_{\mathrm{loc}}}
\left[
2+(1+\delta_{\max})
\left(1+\frac{L_{\max}}{\mu_{\mathrm{loc}}}\right)
\right]
\norm{\nabla h(x_{\mathrm{in}})}
\le
\frac12\Gamma\norm{\nabla h(x_{\mathrm{in}})}
\end{equation*}
of $x^\dagger$. Hence all warm queries lie in $\mathcal{U}$. The warm part of the proof of \Cref{prop:proximal-tracking}, with the budgets \eqref{eq:tracking-budgets}, restores \eqref{eq:relative-proximal} at $x_{\mathrm{new}}$, thereby establishing all four induction conditions at the next anchor.

It follows inductively that every oracle query remains in $\mathcal{U}$. Since local tracking and descent hold at every update, the residual half-life argument \eqref{eq:epoch-bound}--\eqref{eq:residual-half-life} and query accounting \eqref{eq:exact-query-bound} apply unchanged. This proves both conclusions of \Cref{lem:spatial-confinement}.
\end{proof}
\endgroup

\section*{Statements and Declarations}

\paragraph{Funding.}
Lei Zhang was supported by the National Natural Science Foundation of China (Grant No.~12225102, T2321001, and 12288101). Jin Zhao was supported by the Beijing Natural Science Foundation (Grant No. JR25003) and the National Natural Science Foundation of China (Grant No. 12671504).

\paragraph{Author contributions.}
Hua Su made the principal contribution to the conception, theoretical development, and writing of the study. Lei Zhang and Jin Zhao contributed to the development and critical revision of the work. All authors read and approved the final manuscript.

\paragraph{Competing interests.}
The authors have no relevant financial or non-financial interests to disclose.

\paragraph{Availability of data and materials.}
Data availability is not applicable to this article because no datasets were generated or analyzed during this study.

\bibliographystyle{spmpsci}
\bibliography{references}

\end{document}